\documentclass{siamart190516}
    \usepackage[utf8]{inputenc}
    
    \usepackage{amsmath,amssymb,amsfonts}
    \usepackage{tikz}
    \usepackage{cite}
\usepackage{ upgreek }
\usepackage{float}
\usepackage{url}
\usepackage{subcaption}
\usepackage{mathtools}
\usepackage{ifthen}
\usepackage{enumitem}
\usepackage{epsfig}
\usepackage{epstopdf}

\usepackage{fonts}
\usepackage{color}
\definecolor{orange}{rgb}{0.9,0.4,0}
\definecolor{darkgreen}{rgb}{0.1,0.45,0.1}
\definecolor{purple}{rgb}{0.6,0,0.6}

\usepackage{xfrac}
\newtheorem{assumption}{Assumption}

\newcommand{\trace}{T}
\newcommand{\psolve}{F}
\newcommand{\intv}{P}
\newcommand{\map}{N}
\newcommand{\mmap}{\widetilde{N}}
\newcommand{\contractmap}{R}
\newcommand{\bigchi}{\raisebox{2pt}{\mbox{\large$\chi$}}}
\newcommand{\potbound}{\Psi}

\newcommand{\eqnnum}{}
\def\one{(1)}
\def\two{(2)}
\newcommand{\fo}[1][]{f^{\one{#1}}}
\newcommand{\ft}[1][]{f^{\two{#1}}}
\newcommand{\gone}{g^{\one}}
\newcommand{\gtwo}{g^{\two}}
\newcommand{\rone}{r^{\one}}
\newcommand{\rtwo}{r^{\two}}

\newcommand{\uone}{u^{\one}}
\newcommand{\utwo}{u^{\two}}
\newcommand{\Mone}{M^{\one}}
\newcommand{\Mtwo}{M^{\two}}
\newcommand{\qone}{q^{\one}}
\newcommand{\qtwo}{q^{\two}}
\newcommand{\comp}{\widetilde{Z}}
\newcommand{\compv}{\widetilde{V}}
\newcommand{\compvv}{\widetilde{V}}
\newcommand{\normweight}{\eta_{\veps}}
\newcommand{\edummy}{\Xi}

\newcommand{\fot}[1][]{\widetilde{f}^{(1){#1}}}
\newcommand{\ftt}[1][]{\widetilde{f}^{(2){#1}}}
\newcommand{\ftwotilde}[1][]{
	\ifthenelse{\equal{#1}{}}
	   {\ftt}
	   {\ftt[,#1]}}
\newcommand{\fonetilde}[1][]{
	\ifthenelse{\equal{#1}{}}
	   {\fot}
	   {\fot[,#1]}}
	   
\newcommand{\ftwo}[1][]{
	\ifthenelse{\equal{#1}{}}
	   {\ft}
	   {\ft[,#1]}}
\newcommand{\fone}[1][]{
	\ifthenelse{\equal{#1}{}}
	   {\fo}
	   {\fo[,#1]}}

\newcommand{\veps}{\varepsilon}
\newtheorem{remark}{Remark}

\title{Analysis of a new implicit solver for a semiconductor model
\thanks{This material is based upon work supported by the U.S. Department of Energy, Office of Science, Office of Advanced Scientific Computing Research, as part of their Applied Mathematics Research Program. The work was performed at the Oak Ridge National Laboratory, which is managed by UT-Battelle, LLC under Contract No. De-AC05-00OR22725. The United States Government retains and the publisher, by accepting the article for publication, acknowledges that the United States Government retains a non-exclusive, paid-up, irrevocable, world-wide license to publish or reproduce the published form of this manuscript, or allow others to do so, for the United States Government purposes. The Department of Energy will provide public access to these results of federally sponsored research in accordance with the DOE Public Access Plan (\texttt{http://energy.gov/downloads/doe-public-access-plan}).
}}
\author{Victor P. DeCaria\thanks{Computational Mathematics Group, Computer Science and Mathematics Division, Oak Ridge National Laboratory, Oak Ridge, TN 37831, USA(\email{decariavp@ornl.gov}).}
\and Cory D. Hauck\thanks{Computational Mathematics Group, Computer Science and Mathematics Division, Oak Ridge National Laboratory, Oak Ridge, TN 37831, USA and 
Mathematics Department, University of Tennessee, Knoxville, TN 37996, USA (\email{hauckc@ornl.gov}).} 
\and M. Paul Laiu\thanks{Computational Mathematics Group, Computer Science and Mathematics Division, Oak Ridge National Laboratory, Oak Ridge, TN 37831, USA  (\email{laiump@ornl.gov}).}}

\begin{document}
\maketitle

\begin{abstract}
We present and analyze a new iterative solver for implicit discretizations of a simplified Boltzmann-Poisson system. The algorithm builds on recent work that incorporated a sweeping algorithm for the Vlasov-Poisson equations as part of nested inner-outer iterative solvers for the Boltzmann-Poisson equations. The new method eliminates the need for nesting and requires only one transport sweep per iteration. It arises as a new fixed-point formulation of the discretized system which we prove to be contractive for a given electric potential. We also derive an accelerator to improve the convergence rate for systems in the drift-diffusion regime. We numerically compare the efficiency of the new solver, with and without acceleration, with a recently developed nested iterative solver.
\end{abstract}
\section{Introduction}
Electron transport in semiconductors with negligible electron-electron iteration can be modeled by a simplified Boltzmann-Poisson system of equations \cite{markowich2012semiconductor} of the form
\begin{subequations}\label{eqn:model}
\begin{gather}
\varepsilon \frac{\partial f_{\varepsilon}}{\partial  t}   + v \cdot\nabla_xf_{\varepsilon} + \nabla_x \Phi_{\varepsilon} \cdot \nabla_vf_{\varepsilon} = \frac{\omega}{\varepsilon} (M\rho_{\veps} - f_{\veps}) + \varepsilon q, 
\label{eqn:model_vlasov} \\
\Delta \Phi_{\varepsilon} = \rho_{\veps} - D ,\quad M=M_{\Theta}(v)= (2\pi\Theta)^{-d/2}e^{-v^2/2\Theta}, \quad \rho_{\veps} = \int_V f_{\veps}\, dv, \label{eqn:model_poisson}  \\
f_{\varepsilon}|_{t = 0} = f^0, \quad f_{\veps}|_{\partial Z_-} = f_{-},  \quad  \Phi_{\veps}|_{\partial X} = \potbound. \label{eqn:icbc}
\end{gather}
\end{subequations}
Here $f_{\veps} = f_{\veps} (x,v,t)$ denotes the electron distribution at position $x\in X \subset \bbR^d$, velocity $v\in V=\bbR^d$, and time $t\in\bbR^+$; $\Phi_{\veps} = \Phi_{\veps}(x,t)$ denotes the electric potential; and $\nabla_x \Phi_{\veps}$ is the electric field.  
The linear operator $f_{\veps} \mapsto M\rho_{\veps} -f_{\veps}$ is a simplified collision operator that models electron scattering with the semiconductor background, a process which drives $f_{\veps}$ towards the local equilibrium state $M \rho_{\veps}$, where $M$ is a Maxwellian whose  temperature is given by  the constant background lattice temperature $\Theta >0$. The strength of scattering is determined by the parameter $\veps>0$, while variations  in space and time are specified by the (scaled) collision frequency $\omega = \omega(x,t) \in [0,1]$.  In some situations, we drop the superscript $\varepsilon$ when there is no confusion.  

The external volumetric source $q$ is assumed to be given, as are the initial data $f^0$ and inflow data $f_{-}$.  The latter is defined on the inflow boundary $\partial Z_-$ of the phase space $Z=X\times V$:
\label{eqn:bdrys}
\begin{gather}
\partial Z_- = \{z \in  \partial Z ~|~ a(z) \cdot n_z(z) < 0\},
\end{gather}
where $n_z(z)$ is the outward unit normal to $Z$ at $z = (x,v)$ and $a(z) = (v,\nabla_x \Phi_{\veps}(x,t))$.
The outflow boundary $\partial Z_+$ is defined analogously with $z$ such that $a(z) \cdot n_z(z) > 0$.

The Poisson equation in \eqref{eqn:model_poisson} couples $\Phi_\veps$ to $\rho_\veps$ and the doping profile $D = D(x)$, which is fixed in time. Dirichlet boundary conditions for $\Phi_{\veps}$ are given by a function $\potbound$.

It is difficult to construct general purpose methods for \eqref{eqn:model}, in part because there may be large spatio-temporal variations in $\omega$.  When $\omega \approx  1$ and $\veps$ is small, $f_\veps \approx M \rho_0$, where $\rho_0$ satisfies a drift-diffusion equation that is independent of $\veps$ \cite{masmoudi2007diffusion,abdallah2004}.  In such cases, various semi-implicit strategies can be used to correctly achieve this limit; see for example \cite{jin2000discretization,schmeiser1998convergence, dimarco2014implicit}.   However if $\omega \equiv 0$  and $\veps$ is small, then \eqref{eqn:model} reverts to a stiff Vlasov-Poisson system.  The challenge of simulating \eqref{eqn:model} in both settings simultaneously was discussed in detail in \cite{laiu2019fast} where a fully implicit time discretization was proposed.  A solver strategy was then constructed using three basic ingredients, all of which were inspired by approaches developed for radiation transport, which is often simpler because the particles are neutral \cite{Adams-Larsen-2002,larsen2010advances}.  The first ingredient is a fixed-point formulation in terms of the variable $\rho$ instead of $f$, which significantly reduces the memory footprint for iterative methods that use multiple copies of data, such as Krylov methods and Anderson Acceleration.   The second ingredient is a sweeping strategy that inverts the operator in \eqref{eqn:model} under the assumption that $\rho$ is fixed.  This strategy was developed in \cite{Garrett2018AFS} and, unlike the radiation transport case, requires a special domain decomposition to handle the fact that characteristics of \eqref{eqn:model} may form cycles in phase space.  The resulting method then iterates over unknowns on the boundary between domains, rather than the full phase space.  The third ingredient is a preconditioner to improve efficiency near the drift-diffusion limit, in which case the  drift-diffusion equation for $\rho_0$ is the natural choice.

In the present paper, we first perform temporal stability analysis to help justify the fully implicit approach. In particular, we provide a continuous proof of weighted $L^2$ stability. This result  is similar to the one in \cite{schmeiser1998convergence} in that it is does not degenerate as $\veps \to 0$.  However, the norm used here is independent of the potential and can therefor be easily generalized to algebraically stable time stepping methods with a time dependent electric field.    

We then develop a new, tightly coupled iterative method that improves on the nested approach taken in \cite{laiu2019fast} which used an outer loop to iterate over $\rho$ and an inner loop to iterate over the  boundary unknowns, denoted here by $\widetilde{f}$, via the sweeping procedure.  Instead, we formulate a fixed-point strategy for the couple $(\rho, \widetilde{f})$.  The motivation for this strategy is that, since sweeping is the dominant computational cost, we should extract as much work from one sweep as possible.  Thus, once a sweep provides a guess for $f$ on the entire phase space, both $\rho$ and $\widetilde{f}$ are updated. The couple $(\rho, \widetilde{f})$ still has a smaller footprint than the full phase space, making the approach amenable to  Krylov methods or Anderson Acceleration which usually converge faster than fixed-point iterations.  Using the stability estimates in \cite{schmeiser1998convergence}, we prove that the new formulation is a contraction mapping in the linear case of a given electric field, thereby guaranteeing convergence of the solver.  Numerical tests for one-dimensional geometries show that methods based on the new formulation are around three to five times faster than the nested iterative approach.

The remainder of this paper is organized as follows. In Section \ref{sec:temporal_stability}, we analyze the stability of the temporally discretized system for  both implicit Euler and the second order backward differentiation formula (BDF2). In Section \ref{sec:var_form}, we introduce the phase space discretization and formulate the resulting system as a lower dimensional fixed-point problem. We then prove that for a prescribed electric field, the fixed-point map is a contraction. In Section \ref{sec:solver_strategies}, we give a brief overview of the nonlinear solver strategies. Namely, we recall Anderson Acceleration for fixed-point maps, and we derive a new drift diffusion accelerator. Numerical results are presented in Section \ref{sec:numerical_tests}.

\section{Temporal stability\label{sec:temporal_stability}}

The goal of this section is to provide stability results that support the use of implicit time discretizations of \eqref{eqn:model}.  At the continuous level, stability of the entropy density $f\log(f /M) - f$ was established in \cite{Levermore1998MomentCH}.  While extending this result to an implicit Euler discretization is straight-forward, doing so for  higher-order time stepping schemes is less obvious, unless one resorts to the nonlinear, space-time Galerkin framework proposed in \cite{BARTH20063311}.  While elegant, this approach can be very expensive in practice.  

More conventional stability in weighted $L^2$ norms is also challenging.   The main difficulty, as observed for example in \cite{masmoudi2007diffusion}, is that \eqref{eqn:model} contains terms which are stable in different inner product spaces, but not simultaneously. Let 
\begin{equation}
\|f\|= \left(\int_X \int_V |f|^2 dvdx \right)^{1/2}
\quad \text{and} \quad
\|f\|_{M^{-1}} = \left(\int_X \int_V |f|^2 M^{-1} dvdx \right)^{1/2}
\end{equation}
be the standard and weighted $L^2$ norms on $Z$.  While stability is immediate in the extreme cases of pure advection ($\omega \equiv 0$) in the standard $L^2$ norm or with no electric field $\nabla_x \Phi = 0$ in the weighted norm, 
the advection and scattering operators are not monotone in the weighted and standard $L^2$ spaces, respectively.  This challenge was addressed in \cite{schmeiser1998convergence, ringhofer2002} by introducing a weight that depends explicitly on the electric potential, resulting in a time-dependent entropy density $f^2e^{-\Phi}M^{-1}$. A Gr\"onwall estimate independent of $\veps$ is derived under a regularity assumption on $\partial_t \Phi$. However, there is not a straightforward extension to an energy argument in the time discrete setting.  In  Section \ref{sec:continuous_stability} below, we prove an alternative stability estimate in the $\|f\|_{M^{-1}}$ norm that is $\veps$ independent. This estimate can then be easily extended to the time discrete case, as shown in Section \ref{sec:time_discrete_stab}.

\subsection{The continuous case\label{sec:continuous_stability}}

We begin by studying the continuous time case.  Our strategy is based on the following observation:  If  $\veps$ is sufficiently small, the collision kernel becomes dominant, and stability in the weighted $L^2$ space is obtainable. On the other hand, if $\veps$ is large, collisions become insignificant so that standard $L^2$ stability is obtained. Thus we effectively glue together bounds from these two regimes to obtain an $\veps$ independent bound.
We make the following assumptions.
\begin{assumption}\label{assump:l2_stability_thm}
The following hold.
\begin{enumerate}[label=(\alph*)]
\item \label{assump:non_vanishing_omega} The collision frequency does not vanish, $0 < \omega_{\min} \leq \omega$. 
\item \label{assump:bdd_field} $\|\nabla_x \Phi\|_{\infty} := \sup_{x,t,\veps} |\nabla_x \Phi_\veps(x,t)| <  \infty$.
\item \label{assump:eps_independent_init} There exists a constant $C_{0}$ (independent of $\veps$) such that $\|f^0\|_{M^{-1}}^2 \leq C_0 < \infty.$ 
\item \label{assump:compact_v} $f_\veps$ has compact support in $v$, and vanishes for $|v| > v_{\rm{max}} > 0$. 
\item \label{assump:zero_x} 
Zero inflow, that is, $f_- = 0$ in \eqref{eqn:icbc}.
\item \label{assump:no_source} The external source is $q=0$.
\end{enumerate}
\end{assumption}
Assumption \ref{assump:l2_stability_thm}\ref{assump:non_vanishing_omega} is generally assumed to obtain the drift-diffusion limit as $\veps\rightarrow 0$ \cite{masmoudi2007diffusion}. The regularity assumption  \ref{assump:l2_stability_thm}\ref{assump:bdd_field} is weaker than what is used to guarantee a mild solution in \cite{schmeiser1998convergence}, but stronger than the $H^1(X)$ regularity established in \cite{masmoudi2007diffusion} in the context of re-normalized  solutions. Assumption \ref{assump:l2_stability_thm}\ref{assump:eps_independent_init} ensures the Gr{\"o}nwall constant is independent of $\veps$. 
Assumption \ref{assump:l2_stability_thm}\ref{assump:compact_v} is reasonable in our setting since the computational domain must always be bounded.  Indeed, most schemes typically set $v_{\rm{max}} \Theta^{-1/2} \lesssim 10$.  However, this assumption is not physical at the continuum level, since the equilibrium solution $M\rho$ has unbounded support.  It is possible to replace Assumption \ref{assump:l2_stability_thm}\ref{assump:compact_v} with a condition on the decay of $f$ with respect to $v$, resulting in more intricate but essentially similar proofs.  We assume \ref{assump:l2_stability_thm}\ref{assump:zero_x} for simplicity, although extension to nonzero incoming data is an interesting problem. Assumption \ref{assump:l2_stability_thm}\ref{assump:no_source} is assumed for simplicity. The inclusion of external source $q$ would simply alter the growth factor in the stability result by an additive constant depending on $\|q\|$.
\begin{theorem}\label{thm:l2_continuous_stability}
Under Assumption \ref{assump:l2_stability_thm}, for all $T>0$, there exists a constant $C_1$ independent of $\veps$ such that
\begin{gather}\label{eqn:energy_about_mean_bound}
\|f_{\veps}(T)\|_{M^{-1}}^2  
\leq
C_0\exp\left(C_1 T\right). 
\end{gather}

\end{theorem}
\begin{proof} We temporarily suppress the $\varepsilon$ subscripts.
Our goal is to show that
\begin{equation}\label{eqn: gron_diff}
 \frac{d}{dt}\|f\|_{M^{-1}}^2< C_1 \|f\|_{M^{-1}}^2,
\end{equation}
from which \eqref{eqn:energy_about_mean_bound} follows.  
Let $\compvv = \{v \in V\,\,|\,\, |v| \leq v_{\max}\}$ and $g = f - M \rho$. Multiplication of \eqref{eqn:model_vlasov} by $f/M$ and integration over $v \in V$ and $x \in X$ gives
\begin{equation}\label{eqn:weighted energy}
\veps\frac{d}{dt}\|f\|_{M^{-1}}^2
+  \int_X \nabla_x \Phi \cdot \left(\int_{\compvv} \frac{\nabla_v f^2}{M}  dv \right) dx 
+ \int_{X} \frac{2\omega}{\veps} \left(\int_{V}\frac{g^2}{M}dv\right)dx \leq 0. 
\end{equation}
The inequality arises from neglecting $\int_{\partial Z_+}v\cdot n_xf^2M^{-1}ds$ from the divergence theorem.
Integration by parts and writing $f = g + M\rho$ gives
\begin{equation}\label{eqn:int_by_parts_v}
\int_{\compvv} \frac{ \nabla_v f^2}{M}  dv 
	=  -\int_{\compvv} \frac{v}{\Theta} \frac{f^2}{M} dv 
	=  -\int_{\compvv} \frac{v}{\Theta}  \frac{g^2}{M} dv - 2 \int_{\compvv} \frac{v\rho M}{\Theta}  \frac{g}{M}dv.
\end{equation}
We consider two cases.  Let
 \begin{equation}
C_1 =  \frac{\|\nabla_x \Phi\|_{\infty}^2 v_{\max}^2}{\Theta^2\omega_{\min}}
\quad \text{and} \quad
C_2 = \frac{\|\nabla_x \Phi\|_{\infty}v_{\max}}{\Theta}.
 \end{equation}

\noindent
\textbf{Case 1:} Assume $\veps > C_2^{-1} \omega_{\text{min}}$.  Then $\frac{C_2}{C_1}= \frac{\omega_{\min}}{ C_2}< \veps$, so the first equality in \eqref{eqn:int_by_parts_v} implies that
\begin{multline} \label{eqn:case_1_bound}
\left| \int_{X} \nabla_x \Phi \cdot \left(\int_{\compvv} \frac{\nabla_v f^2}{M}  dv \right) dx \right| 
=\left| \int_{X} \nabla_x \Phi \cdot \left(\int_{\compvv} \frac{v}{\Theta} \frac{f^2}{M} dv \right) dx \right| \\
 \leq  \frac{\|\nabla_x \Phi\|_{\infty}v_{\max}}{\Theta} \|f\|_{M^{-1}}^2
= C_2 \|f\|_{M^{-1}}^2 
< \veps C_1 \|f\|_{M^{-1}}^2.
\end{multline}
Applying  \eqref{eqn:case_1_bound} to \eqref{eqn:weighted energy} gives \eqref{eqn: gron_diff}.

\noindent
\textbf{Case 2:} Assume $\veps \leq C_2^{-1} \omega_{\text{min}}$.
  By way of  Young's inequality, the second equality in \eqref{eqn:int_by_parts_v} implies that 
\begin{equation}
\begin{aligned}\label{eqn:case_2_bound}
& \left| \int_{X} \nabla_x \Phi \cdot \left(\int_{\compvv} \frac{\nabla_v f^2}{M}  dv \right) dx \right|  \\
 \leq & \left |\int_{X} \nabla_x \Phi \cdot \left(\int_{\compvv} \frac{v}{\Theta}  \frac{g ^2}{M} dv\right) dx  \right|
+\left |\int_X \nabla_x \Phi \cdot \left(	2 \int_{\compvv} \frac{v\rho M^{1/2}}{\Theta}  \frac{g }{M^{1/2}}dv\right) dx \right|\\
\leq & \frac{\|\nabla_x \Phi\|_{\infty}v_{\max}}{\Theta} \|g \|_{M^{-1}}^2 
	+ \frac{\|\nabla_x \Phi\|_{\infty}v_{\max}}{\Theta} \left(\frac{C_2 \veps}{\omega_{\text{min}}} \int_{X} \int_{\compvv} \rho^2 M dv dx + \frac{\omega_{\text{min}}} {C_2  \veps}\|g \|_{M^{-1}}^2\right)  \\
 \leq & \frac{C^2_2  \veps }{\omega_{\text{min}}} \int_X \rho^2 dx 
	+ \left(C_2+ \frac{\omega_{\text{min}}} {\veps}\right )\|g \|_{M^{-1}}^2 
	\leq   C_1 \veps \|f \|_{M^{-1}}^2
	+ \left(C_2+ \frac{\omega_{\text{min}}} {\veps}\right )\|g \|_{M^{-1}}^2,
\end{aligned}
\end{equation}
where in the last line, we have used the fact that $C_1 = \frac{C^2_2  }{\omega_{\text{min}}}$ and  $\int_X \rho^2dx\leq\|f\|_{M^{-1}}^2$.
Applying the bound from \eqref{eqn:case_2_bound} to \eqref{eqn:weighted energy} and gathering terms gives
\begin{gather}
\veps \frac{d}{dt}\|f\|_{M^{-1}}^2
+  \left(\frac{\omega_{\text{min}}}{\veps}  - C_2\right)\|g \|_{M^{-1}}^2
 \leq C_1 \veps \|{f} \|_{M^{-1}}^2,
\end{gather}
from which \eqref{eqn: gron_diff} follows.
\end{proof}

\begin{remark}
From the proof of Theorem \ref{thm:l2_continuous_stability}, we can also show for $\varepsilon$ sufficiently small that
 \begin{gather*}
 \sqrt{ \int_{0}^T \|f_{\veps} - M\rho_{\veps}\|_{M^{-1}}^2 dt} \leq C \veps,
\end{gather*}
where $C = C(T, \omega, \Phi, \Theta, f^0)$, which is consistent with the expected behavior in the drift-diffusion limit.
\end{remark}

\begin{remark}
While a result similar to Theorem \ref{thm:l2_continuous_stability} was shown in \cite{schmeiser1998convergence}, the energy therein had $\Phi$-dependent weights. As a result, $\Phi$ was assumed to be fixed in time to prove stability and convergence in the time discrete case. The proof of Theorem~\ref{thm:l2_continuous_stability} avoids a $\Phi$-dependent integrating factor, which means it can be modified with $G$-Stability analysis \cite{dahlquist1978} for linear multistep methods, as demonstrated in the next section.
\end{remark}

\subsection{The time discrete case\label{sec:time_discrete_stab}}

In this section, we analyze the properties of the discrete time, continuous space problem to motivate the use of fully implicit methods. The result does not follow immediately when the phase space is discretized unless the spatial discretization is tailored to preserve the result, or extra assumptions are applied. However, the analysis provides insight into the expected stability for the fully discrete system.

Given $\Delta t$, we define $t^n = n\Delta t$ and $f^{n}$ to be an approximation of $f(t^n)$. The known source at time $t^n$ is denoted $q^{n}$. The main time discrete equations we will consider are implicit Euler and BDF2, which are given in order as
\begin{equation}\label{eqn:time_discrete}
	\veps\frac{f^{n+1}-f^{n}}{\Delta t}  + v \cdot\nabla_xf^{n+1} + \nabla_x \Phi^{n+1} \cdot \nabla_vf^{n+1} = \frac{\omega}{\veps}(M\rho^{n+1} - f^{n+1}) + \varepsilon q^{n+1},
\end{equation}
and
\begin{equation}\label{eqn:time_discrete_bdf2}
	\veps\frac{3f^{n+1}-4f^{n} + f^{n-1}}{2\Delta t}  + v \cdot\nabla_xf^{n+1} + \nabla_x \Phi^{n+1} \cdot \nabla_vf^{n+1} = \frac{\omega}{\veps}(M\rho^{n+1} - f^{n+1}) + \varepsilon q^{n+1},
\end{equation}
where, in both cases,
\begin{equation}\eqnnum
	\Delta \Phi^{n+1} = \rho^{n+1} - D,\qquad \rho^{n+1} = \int f^{n+1} dv.
\end{equation}
BDF2 is chosen since it is both $L$-Stable and $G$-Stable \cite{wanner1996solving}, which is important in the infinitely stiff limit as $\veps\rightarrow 0$.  However, other reasonable choices exist.

\begin{theorem}[Stability of implicit Euler]\label{thm:l2_stab}
Let $T=N\Delta t$ be the final time. Suppose that $f^n$ solves \eqref{eqn:time_discrete} for all $n \in \{0,1,...,N\}$. Under Assumption \ref{assump:l2_stability_thm} and using the same constants $C_0$, $C_1$, and $C_2$ as in Section \ref{sec:continuous_stability}, $\Delta t < C_1^{-1}$ implies
\begin{equation}\label{eqn:gronwall_discrete}
\frac{1}{2}\|f^n\|_{M^{-1}}^2 \leq C_0\left(1 - C_1 \Delta t  \right)^{-N}. 
\end{equation}

\end{theorem}
\begin{proof}

After multiplying \eqref{eqn:time_discrete} by $f^{n+1}/M$ and integrating in $x$ and $v$, the proof proceeds essentially the same as the proof of Theorem \ref{thm:l2_continuous_stability}.  The only substantial differences are the treatment of time differences and the use of a discrete Gr{\"o}nwall lemma. The time differences for implicit Euler is dealt with using

$$ \left(\frac{f^{n+1}-f^n}{\Delta t}\right)f^{n+1} \geq \frac{1}{\Delta t}\left(\frac{1}{2}(f^{n+1})^2-\frac{1}{2}(f^{n})^2\right). $$
We then apply the discrete Gr{\"o}nwall lemma, \cite[Proposition 3.1]{emmrichdiscrete}, from which the $\mathcal{O}(1)$ time step condition arises.
\end{proof}

Note that as $\Delta t \rightarrow 0$, the right hand side of \eqref{eqn:gronwall_discrete} approaches the right hand side of \eqref{eqn:energy_about_mean_bound}. While we only state the result and proof for implicit Euler, we note that the technique easily extends to other $G$-Stable methods (such as BDF2 in \eqref{eqn:time_discrete_bdf2}).

\section{The new iterative solver\label{sec:new_solver_formulation}}

In this section, we present a new iterative solver for \eqref{eqn:time_discrete} which improves on the previous approach in \cite{laiu2019fast} using a more tightly coupled strategy that reduces the number of required inversions of the phase space advection operator.  We also present a convergence proof for the solver under the assumption of a fixed electric field.  First, however, we quickly review the discontinuous Galerkin (DG) discretization of \eqref{eqn:time_discrete} to which the solver is applied.  Since the discretization has already been described in detail in \cite{laiu2019fast,Garrett2018AFS}, our presentation will be brief.  Generally speaking, many other discretizations of the phase space can be used.  The only substantive requirements are (i) an upwind direction is well defined for a fixed electric field and (ii) only upwind information is used to approximate derivatives.   These requirements allow for the use of the sweeping algorithm developed in \cite{Garrett2018AFS}, but can be relaxed if a different strategy is used to invert the advection operator.  In addition, the solver presented below can be applied to higher-order time discretization schemes such as diagonally implicit Runge-Kutta (RK) methods and linear multistep methods (LMMs). In some of the numerical tests, we use BDF2.

\subsection{Discontinuous Galerkin discretization\label{sec:var_form}}
 We restrict the velocity to a bounded domain $\compv \subset{V}$ and discretize the computational domain $\comp = X \times \compv$ with a Cartesian grid of open cells $K= K^x \times K^v$ of uniform size $\Delta x \times \Delta v$. Let $\mathcal T_h$ be the set of all such cells, with $h = \max\{\Delta x, \Delta v\}$; let $\mathcal F^{\rm{int}}_h$ be the set of cell edges in the interior of the domain; 
let $\mathcal F^{\rm{+/-}}_h$ be the set of cell edges in the outgoing/incoming boundary $\partial \comp_{\rm{+}/\rm{-}}$ of the computation domain; 
and let $\mathcal F_h = \mathcal F^{\rm{int}}_h \cup \mathcal F^{-}_h \cup \mathcal F^{\rm{+}}_h$. We associate a positive normal direction $n_e$ to each $e \in \mathcal{F}_h$, with the convention that  $n_e$ be the outward normal on $\mathcal F^{-}_h \cup \mathcal F^{\rm{+}}_h$. For the DG spaces, let $\cG_h^p(\comp)$ (resp. $\cG_h^p(X)$) be the set of all piecewise continuous functions on $\comp$ (resp. $X$) that are order $p$ polynomials in each cell $K$ (resp. $K^x$). Let $\cW_h^x$ be the space of globally continuous functions that are linear in each spatial cell $K^x$. We suppress the $\veps$ subscripts on $f$, $\rho$, and $\Phi$ in this section to simplify notation. While Theorem \ref{thm:convergence} below is valid for any continuous approximation of $\Phi$, we only use piecewise linear elements in our tests. The reason being that, when the approximate potential $\Phi_h$ is of order higher than linear, the electric field $E_h:=\nabla_x \Phi_h$ may change sign in a spatial cell, which complicates the sweeping procedure borrowed from \cite{Garrett2018AFS}. While an analysis of the Vlasov-Poisson equation \cite{ayuso2011discontinuous} requires a piecewise quadratic electric field for overall second-order convergence, we observe second-order convergence for a manufactured solution in Section \ref{sec:convergence_rate}. Implementing the sweeping strategy efficiently for higher order elements remains an open problem.

The sweeping algorithm developed in \cite{Garrett2018AFS} uses upwind traces on the cell edges. For $g_h \in \cG_h^1(\comp)$ and $z \in e \in \mathcal F_h$, let $g_h^{\pm}(z) = \lim_{\delta \rightarrow 0^+}g_h(z \pm \delta n_e)$, where $z = (x,v)$. Denote jumps across the interfaces by
$[g_h] = g_h^+-g_h^-$
and averages by
$\langle g_h \rangle = \frac{1}{2}(g_h^+ + g_h^-)$. The upwind trace is then defined as
\begin{equation}
\hat{g}_h(g_h^+,g_h^-,E_h) = \langle g_h\rangle - \frac{1}{2}\operatorname{sgn}(a_h \cdot n_e)[g_h],\quad\text{with}\quad a_h = (v,E_h).
\end{equation}
With these conventions, the DG discretization of \eqref{eqn:time_discrete} takes the compact form:  \textit{Find $(f_h^{n+1},\Phi_h^{n+1}) \in \cG_h^1(\comp) \times \cW_h^x$ such that, for all $(g_h,w_h) \in \cG_h^1(\comp) \times \cW_h^x$,}
\begin{subequations}
\begin{gather}
\mathcal{C}_{E_h^{n+1}}(f_h^{n+1},g_h) -  \cQ(f_h^{n+1},g_h)
= \mathcal{L}(g_h), \label{eqn:imp_euler_fully_discrete}
\\
E_h^{n+1} = \nabla_x \Phi_h^{n+1}, \quad \int_X \nabla_x \Phi_h^{n+1} \cdot \nabla_x w_h \,dx  = -\int_X(\rho_h^{n+1} - D) w_h \,dx, \label{eqn:poisson_discrete}
\end{gather}
\end{subequations}
where the operator $\mathcal{C}_{E_h}$ collects terms from the discretization of the gradient:
\begin{equation}
\mathcal{C}_{E_h}(f_h,g_h)  
= -\sum_{K \in \mathcal T_h} \int_{K} a_h f_h \cdot \nabla_z g_h  \,dvdx
+ \sum_{e \in  \mathcal F_h^{\rm{int}} \cup \mathcal F_h^{\rm{+}} } \int_{e}a_h \hat{f}_h [g_h] \cdot n_e  \,ds(x,v),
\end{equation}
the operator $Q$ collects terms from the collision operator plus the implicit term in the temporal discretization:
\begin{equation}\label{eqn:operatorQ}
\cQ(f_h,g_h) 
= \underbrace{\sum_{K \in \mathcal T_h}\int_{K} \frac{\omega}{\veps} M(v)\rho_hg_h \,dvdx}_{\mathcal{S}(\rho_h,g_h) }
- \underbrace{\sum_{K \in \mathcal T_h}\int_{K}\left( \frac{\veps}{\Delta t}+ \frac{\omega}{\veps} \right) f_hg_h \,dvdx}_{\mathcal{R}(f_h,g_h)},
\end{equation} 
and the operator $\mathcal{L}$ combines the volumetric source, incoming boundary conditions, and the explicit term in the temporal discretization:
\begin{equation}
\mathcal{L}(g_h) 
\label{eq:L}
=  \sum_{K \in \mathcal T_h}\int_{K} \veps\left(q^{n+1} + \Delta t^{-1} f^n_h\right) g_h\,dvdx 
+ \sum_{e \in  \mathcal F_h^{\rm{-}}} \int_{e} a_h f_{-} g_h^- \cdot n_e  \,ds(x,v). 
\end{equation} 

\subsection{Formulation as a fixed point problem\label{sec:formulation_of_fixed_point}}
To lighten the notation, we remove the superscript $n+1$, setting  $f_h := f_h^{n+1}$, $\rho_h := \rho_h^{n+1}$, and $E_h := E_h^{n+1}$.  We then reformulate \eqref{eqn:imp_euler_fully_discrete} to isolate $\rho_h$.  There are two reasons for this:  first, the calculation of $\rho_h$ creates global coupling in velocity and second, expressing \eqref{eqn:imp_euler_fully_discrete} in terms of $\rho_h$ allows for a significant reduction in memory costs for Krylov subspace methods.  Let $\psolve: \rho_h \mapsto E_h$ denote the solution map from the density to the electric field defined by Poisson system in \eqref{eqn:poisson_discrete} so that $E_h = \psolve (\rho_h)$; and let $\intv:\cG_h^1(\comp) \rightarrow \cG_h^1(X)$ denote the velocity integral over $\compv$, so that \eqref{eqn:poisson_discrete} can be denoted as $E_h = \psolve(\intv(f_h))$. Then \eqref{eqn:imp_euler_fully_discrete} can be written as 
\begin{equation}\label{eqn:var_form_compact}
\cA_{\psolve(\intv(f_h))}(f_h,g_h) = \mathcal{L}(g_h)+\cS(\intv(f_h),g_h),
\end{equation}
where $\mathcal{A}_{E_h} = \cC_{E_h} + \cR $ and the forms $\cS$ and $\cR$ are defined in \eqref{eqn:operatorQ}.

The main computational kernel for solving \eqref{eqn:var_form_compact} is the sweeping algorithm developed in \cite{Garrett2018AFS}, which for a fixed value of $E_h$, inverts (the linear operator associated to) $\mathcal{A}_{E_h}$.  The sweeping algorithm builds upon well-known methods  for neutral particle transport (see for example \cite{larsen2010advances,Adams-Larsen-2002}), which update information following characteristics in phase space.  Unlike the neutral particle case, the characteristics of the advection operator in \eqref{eqn:model} may be cyclic.  The main contribution of \cite{Garrett2018AFS} was to introduce a domain decomposition of the phase space into $2^d$ subdomains upon which the sign of each component of $v$ is constant.  This decomposition implicitly assumes that each cell $K^v$ is a subset of one and only one such domain.  To simplify the presentation, we assume $d=1$.  The extension to $d>1$ is straight-forward.  

As in \cite{Garrett2018AFS}, let the two subdomains be $Z^1 = \{ (x,v) \subset \comp \colon v >0\}$ and $Z^2 = \{ (x,v) \subset \comp \colon v <0\}$, and let $B = \partial Z^1 \cap \partial Z^2 = \{ (x,v) \subset \comp \colon v =0\}$.  Then for each $g\in\cG_h^1(\comp)$, define
\begin{equation} \label{eqn:velocity_decomp}
\gone = \bigchi _{\{v>0\}}g 
\quad \text{and} \quad 
\gtwo = \bigchi _{\{v<0\}}g.
\end{equation}
Then 
\begin{equation}
\cA_{E_h}(f_h,g_h) = \cA_{E_h}(\fone_h, \gone_h) + \cA_{E_h}(\ftwo_h, \gtwo_h) + \cA_{E_h}(\fone_h, \gtwo_h) + \cA_{E_h}(\ftwo_h, \gone_h),
\end{equation}
and  the only coupling between the two subdomains occurs at the boundary $B$. To isolate the unknowns there, let  $\fonetilde_h$ and $\ftwotilde_h$ be the numerical trace values of $\fone$ and $\ftwo$, respectively, on $B$:
\begin{equation}	\label{eq:subdomain_traces}
	\fonetilde_h (x) := \bigchi_{\{E_h < 0\}}(x) \hat{f}_h(x,0),\quad\text{and}\quad
	\ftwotilde_h (x) := \bigchi_{\{E_h > 0\}}(x) \hat{f}_h(x,0),
\end{equation}
and define
\begin{subequations}
	\label{eq:subdomain_coupling}
	\begin{align}	\label{eq:subdomain_coupling_B1}
	\mathcal{B}^{\one} (\fonetilde_h,\gtwo_h) 
		&= \mathcal{A}_{E_h}(\fone_h,\gtwo_h) = -  \int_{\{x:E_h < 0\}}|E_h|\hat{f}_h(x,0) g_h^-(x,0) dx \:,
\\
	\mathcal{B}^{\two} (\ftwotilde_h,\gone_h) 
		&= \mathcal{A}_{E_h}(\ftwo_h,\gone_h) = -  \int_{\{x:E_h > 0\}}|E_h|\hat{f}_h(x,0) g_h^+(x,0) dx \:.
			\label{eq:subdomain_coupling_B2} 
	\end{align}
\end{subequations}
We write the numerical trace values on $B$ in a compact operator form as $\widetilde{f}_h = \trace_{E_h} f_h$, where the operator $\trace_E$ is defined as $\trace_E f := \trace^{\one}_E \fone + \trace^{\two}_E \ftwo$ with
\begin{subequations}	\label{eq:subdomain_traces_cont}
	\begin{align}
	&\trace^{\one}_{E}(\fone)(x) := \bigchi_{\{E < 0\}}(x) \lim_{\delta \rightarrow 0^+}\fone(x,\delta),
	\\
	&\trace^{\two}_{E}(\ftwo)(x) := \bigchi_{\{E > 0\}}(x) \lim_{\delta \rightarrow 0^-}\ftwo(x,\delta).
	\end{align}
\end{subequations}
From this definition, we see that $\trace_{\edummy_h}$ maps $\cG_h^1(\comp)$ into $\cG_h^1(X)$ for any $\edummy_h \in \cG_h^0(X)$.  

The algorithm in \cite{Garrett2018AFS} uses \eqref{eq:subdomain_traces} and \eqref{eq:subdomain_coupling} to isolate and remove the coupling between $\fone_h$ and $\ftwo_h$ from the left-hand side of \eqref{eqn:var_form_compact}.   
Specifically, from \eqref{eq:subdomain_traces} and \eqref{eq:subdomain_coupling}, solving for $f_h$ in \eqref{eqn:var_form_compact} can be viewed as finding $u_h \in \cG_h^1(\comp)$ that solves 
\begin{subequations}\label{eqn:var_form_decoupled}
\begin{gather}
\mathcal{A}_{\edummy_h}(\uone_h,\gone_h) = \mathcal{L}(\gone_h)+\mathcal{S}(\sigma_h,\gone_h) - \mathcal{B}^{\two} (\rtwo_h,\gone_h),
\label{eqn:var_form_decoupled_1}\\
\mathcal{A}_{\edummy_h}(\utwo_h,\gtwo_h) = \mathcal{L}(\gtwo_h)+\mathcal{S}(\sigma_h,\gtwo_h) -\mathcal{B}^{\one}  (\rone_h,\gtwo_h),
\label{eqn:var_form_decoupled_2}\\
\text{with}\quad \sigma_h = \intv(u_h), \quad r_h = \trace_{\edummy_h}(u_h), \quad\text{and}\quad \edummy_h = F(\sigma_h).
\label{eqn:var_form_decoupled_trace}
\end{gather}
\end{subequations} 
Let $\mmap:\cG_h^0(X)\times \cG_h^1(X)\times \cG_h^1(X) \rightarrow \cG_h^1(\comp)$ denote the operator that maps $(\edummy_h,\sigma_h,r_h)$ to $u_h$ in \eqref{eqn:var_form_decoupled_1}--\eqref{eqn:var_form_decoupled_2}. That is, $u_h$ is computed by $u_h = \mmap(\edummy_h,\sigma_h,r_h)$, where $(\edummy_h,\sigma_h,r_h)$ solves the coupled, lower dimensional system
\begin{equation}
\sigma_h = \intv(\mmap(\edummy_h,\sigma_h,r_h)),\quad r_h = \trace_{\edummy_h}(\mmap(\edummy_h,\sigma_h,r_h)), \quad \edummy_h = \psolve (\sigma_h).
\end{equation}
The evaluation of $\mmap(\edummy_h,\sigma_h,r_h)$ allows for independent inversion of the linear operator associated to $\mathcal A_{\edummy_h}$ in each subdomain, which is referred to as a \textit{sweep}. We note that due to the trace definitions and the domain decomposition, a sweep does not require assembling the matrix associated with $\mathcal{A}_{\edummy_h}$; it only requires inverting a sequence of $3\times 3$ matrices associated to the variational formulation on each cell $K$ (a linear element has three degrees of freedom in the one space-one velocity dimension case).
In view of \eqref{eqn:var_form_compact}, the solution satisfies $f_h = \map(\rho_h, \widetilde{f}_h):=\mmap(\psolve(\rho_h),\rho_h, \widetilde{f}_h)$ where $(\rho_h, \widetilde{f}_h)$ solves
\begin{equation}\label{eqn:fixed_point_problem}
\begin{pmatrix} 
\rho_h  \\[5pt]
\widetilde{f}_h
\end{pmatrix}
=
\begin{pmatrix} 
\intv (\map(\rho_h, \widetilde{f}_h))   \\ 
\trace_{\psolve(\rho_h)} (\map(\rho_h, \widetilde{f}_h))
\end{pmatrix}.
\end{equation}
This fixed point problem \eqref{eqn:fixed_point_problem} is the basis for the method used in this paper. 
Invariably, the most expensive part of solving \eqref{eqn:fixed_point_problem} is the evaluation of $N(\sigma_h,r_h)$ for a given $(\sigma_h,r_h)$, which requires one sweep in each subdomain.

In \cite{laiu2019fast}, the system \eqref{eqn:fixed_point_problem} was solved with a nested iterative approach.  With Picard iteration, this approach results in the algorithm
\begin{gather}\label{eqn:nested iteration}
\rho_h^{k+1}  
	= \intv \left(\map(\rho_h^{k},\widetilde{f}_h^{k+1})\right)  \:,\\
\widetilde{f}_h^{k+1} = \lim_{\ell \to \infty}\widetilde{f}_h^{k+1,\ell},
\quad
\widetilde{f}_h^{k+1,\ell+1} 
	= \trace_{\psolve(\rho^k_h)} \left(\map (\rho_h^{k}, \widetilde{f}_h^{k+1,\ell})\right) .
\end{gather}
A more explicit summary of this algorithm, which we refer to as nested (NEST), is given in Algorithm \ref{alg:nest}.  

In the current work, we investigate a more tightly coupled strategy that does not involve nested iterations.
In this case, a Picard iteration for \eqref{eqn:fixed_point_problem} takes the form
\begin{equation}\label{eqn:fixed_point_iteration}
\begin{pmatrix} 
  \rho_h^{k+1}  \\[5pt]
  \widetilde{f}_h^{k+1}
\end{pmatrix}
=
\begin{pmatrix}
   \intv (\map(\rho_h^{k}, \widetilde{f}_h^k))     \\ 
   \trace_{\psolve(\rho^k_h)} (\map(\rho_h^{k}, \widetilde{f}_h^k))  
\end{pmatrix}
.
\end{equation}
A more explicit summary of this algorithm, which we refer to as nonlinear sweeping (NLS), is given in Algorithm \ref{alg:nls}.   Its main benefit is that it requires only one evaluation of $\map$ (one sweep in each subdomain) per iteration, which results in far fewer sweeps in the solution procedure.

\begin{algorithm}\caption{Nested algorithm (NEST)}\label{alg:nest} 

\vspace{1mm}

\noindent
\textbf{Outer loop}:  Given $\rho^k_h$, solve for $\rho_h^{k+1}$:
 \begin{gather}\label{eqn:nested_source}
\mathcal{A}_{F(\rho_h^k)}(f_h^{k+1},g_h) = \mathcal{L}(g_h)+\mathcal{S}(\rho_h^k,g_h),\\
 \rho_h^{k+1} = \intv(f_h^{k+1}),
\end{gather} 
where \eqref{eqn:nested_source} is solved with the inner loop.  

\noindent \vspace{1mm}

\textbf{Inner loop}:  Given $\rho^k_h$ and $\widetilde{f}_h^{\ell}$, solve for $\widetilde{f}_h^{\ell+1}$:
\begin{subequations}\label{eqn:nested_sweep}
\begin{gather}
\mathcal{A}_{F(\rho_h^k)}(\fone[\ell+1]_h,\gone_h) = \mathcal{L}(\gone_h)+\mathcal{S}(\rho_h^k,\gone_h) - \mathcal{B}^{\two} (\ftwotilde[\ell]_h,\gone_h),
\label{eqn:nested_sweep_1}\\
\mathcal{A}_{F(\rho_h^k)}(\ftwo[\ell+1]_h,\gtwo_h) = \mathcal{L}(\gtwo_h)+\mathcal{S}(\rho_h^{k},\gtwo_h) -\mathcal{B}^{\one} (\fonetilde[\ell]_h,\gtwo_h),
\label{eqn:nested_sweep_2}\\
\fonetilde[\ell+1]_h = \trace_{F(\rho_h^k)}^{\one}(\fone[\ell+1]_h) , \hspace{10mm} \ftwotilde[\ell+1]_h = \trace_{F(\rho_h^k)}^{\two}(\ftwo[\ell+1]_h).
\label{eqn:nested_sweep_trace}
\end{gather} 
\end{subequations}
When $\widetilde{f}_h^{\ell}$ is sufficiently converged, solve \eqref{eqn:nested_sweep_1}--\eqref{eqn:nested_sweep_2} one more time for $(\fone[\ell + 1]_h,\,\ftwo[\ell + 1]_h)$ and set $f_h^{k+1} = \fone[\ell + 1]_h + \ftwo[\ell + 1]_h$.
\end{algorithm}

\begin{algorithm}\caption{Nonlinear sweeping method (NLS)}\label{alg:nls}
Given $(\rho_h^k, \widetilde{f}_h^k)$, solve
\begin{subequations} \label{eqn:nls-explicit}
\begin{gather}
\mathcal{A}_{F(\rho_h^k)}(\fone[k+1]_h,\gone_h) = \mathcal{L}(\gone_h)+\mathcal{S}(\rho_h^k,\gone_h) - \mathcal{B}^{\two} (\ftwotilde[k]_h,\gone_h),
\label{eqn:nls-explicit_1} \\
\mathcal{A}_{F(\rho_h^k)}(\ftwo[k+1]_h,\gtwo_h) = \mathcal{L}(\gtwo_h)+\mathcal{S}(\rho_h^{k},\gtwo_h) -\mathcal{B}^{\one} (\fonetilde[k]_h ,\gtwo_h),
\label{eqn:nls-explicit_2} \\
\rho_h^{k+1} = \intv(f_h^{k+1}), \,\,\,\,\, \fonetilde[k+1]_h = \trace^{\one}_{F(\rho_h^k)}(\fone[k+1]_h) , \,\,\,\,\,\ftwotilde[k+1]_h = \trace^{\two}_{F(\rho_h^k)}(\ftwo[k+1]_h).
\label{eqn:nls-explicit_trace} 
\end{gather} 
\end{subequations}
\end{algorithm}
\subsection{Convergence of the fixed point map}
We now prove the main result of this paper, which is the convergence of Algorithm \ref{alg:nls}. 
For the sake of simplicity, we assume that the incoming boundary conditions are zero, in which case the second term in \eqref{eq:L} can be removed. We also restrict ourselves to the linear case where the electric potential $\Phi_h\in \cW_h^x$ is fixed. The electric field, $E_h = \nabla \Phi_h$, is also fixed as a result. However, we update $\Phi_h$ every iteration in our experiments so that $\Phi$ is self consistently coupled to the Boltzmann equation. Convergence of the nonlinear case remains open.

As in the energy analysis, we need to work with weighted $L^2$ spaces to avoid restrictive conditions arising from $\varepsilon$. The phase space discretization complicates the analysis since we wish to test with functions outside the span of the polynomial basis. Therefore, we test with a projection of the desired test function. Let $\Pi_h$ denote the projection from $L^2(\comp) \rightarrow \mathcal{G}_h^1(\comp)$. 

We prove in Theorem \ref{thm:convergence} that the fixed point method converges to the order of the consistency error arising from using the projection into the DG space. For the remainder of this section, let $\kappa$ be a constant such that $0 < \kappa < 1$. The idea is that if $b$ is small, then $a^{n}$ is almost a contractive sequence if
\begin{equation}\eqnnum
a^{n+1} \leq \kappa a^n + b,
\end{equation}
which implies $a^{n+1} \leq \kappa^n a_0 + \mathcal{O}(b)$. In the current context, $b$ represents the following consistency error.
\begin{definition}
Let $h = \max\{\Delta x, \Delta v\}$, $u_h \in \cG_h^1(\comp)$, $\sigma_h$ and $r_h \in \cG_h^1(X)$. The projection error $\phi_h$ is defined
\begin{equation}
\phi_h = \phi_h^{(1)} + \phi_h^{(2)}, \quad \phi^{(i)}_h = \Pi_h\left(\frac{u^{(i)}_h}{M}e^{-\Phi_h/\Theta}\right) - \frac{u_h^{(i)}}{M}e^{-\Phi_h/\Theta}, \quad i=1,2.
\end{equation}
The consistency error operator is
\begin{equation}\label{eqn:tau_clean}
\begin{alignedat}{2}
\tau_h(u_h, \sigma_h, r_h) = &\mathcal{A}_{E_h} (\uone_h,\phi^{\one}_h ) + \mathcal{A}_{E_h} (\utwo_h,\phi^{\two}_h ) - \mathcal{S}\left( \sigma_h , \phi_h \right) \\
&+\mathcal{B}^{\two} ( \rtwo_h,\phi_h^{\one}) +\mathcal{B}^{\one} ( \rone_h,\phi_h^{\two}). 
\end{alignedat}
\end{equation}
\end{definition}
Put differently, the consistency error is the summed residual of \eqref{eqn:nls-explicit_1} and \eqref{eqn:nls-explicit_2} with $\mathcal{L} = 0$, and $g_h = \phi_h$. We note that it is possible to use the unweighted $L^2(\comp)$ space without a consistency error term, but it requires an unacceptable $\Delta t < C \varepsilon^2$ condition, so we do not show this analysis herein. It is also possible to use analysis similar to the stability proof in Section \ref{sec:continuous_stability} to obtain convergence in the $L^2(\comp)$ norm weighted by $M^{-1}$ under a $\Delta t < C$ condition with a consistency term depending only on $\Delta v$ instead of both $\Delta v$ and $\Delta x$. For brevity of the manuscript, we use the Hamiltonian of the PDE with both $x$ and $v$ dependent weights.

\begin{definition}[Norms and contraction constants for NLS\label{def:norm_contract}]
Define $\normweight(x) = \frac{2\varepsilon^2}{\Delta t} + \omega(x) - 2\varepsilon^{5/2}$ and $M_0=M(0)$. Assume $\omega \in L^{\infty}(X)$, and let $C$ be the constant for the inverse inequality $C \Delta v\|\trace_{E_h}u_h\|^2 \leq \|\partial_v u_h\|^2$. For all $\varepsilon>0$, we define the norm and the estimate for the contraction constant
\begin{subequations}\label{eqn:norm_def_1}
\begin{gather}
\|(\sigma,r)\|_{_{\emph{NLS}_1}}^2 :=\frac{1}{2}\int_X \left\lbrace
   \left(\frac{\varepsilon^2}{\Delta t} + \omega \right)\sigma^2  
   +
   \veps\left(\frac{C\varepsilon\Delta v}{\Delta t} + |E_h| \right)
    \frac{r^2}{M_0} \right\rbrace e^{-\frac{\Phi_h}{\Theta}} dx,
    \\
\kappa_{_{\emph{NLS}_1}} := \left( \max \left( \frac{\omega_{\max}}{\veps^2\Delta t^{-1} + \omega_{\max}},\frac{\|E_h\|_{L^\infty}}{C\Delta v\veps\Delta t^{-1} + \|E_h\|_{L^\infty}} \right) \right)^{1/2}.\label{eqn:norm_def_1_const}
\end{gather}
\end{subequations}
If $\min_x \normweight >0$, we define the norm and the estimate for the contraction constant
\begin{subequations}\label{eqn:norm_def_2}
\begin{gather}
\|(\sigma,r)\|_{_{\emph{NLS}_2}}^2 :=\frac{1}{2}\int_X \left\lbrace
   \normweight(x)\sigma^2  
   +
   \veps\left(2 C\varepsilon^{3/2}\Delta v + |E_h| \right)
    \frac{r^2}{M_0} \right\rbrace e^{-\frac{\Phi_h}{\Theta}} dx,
    \\
\kappa_{_{\emph{NLS}_2}} :=\left( \max \left( \frac{\omega_{\max}}{2\varepsilon^2\Delta t^{-1} + \omega_{\max}   - 2\varepsilon^{5/2}}\,,\,\,\,\,\frac{\|E_h\|_{L^\infty}}{C\Delta v \varepsilon^{3/2} + \|E_h\|_{L^\infty}} \right) \right)^{1/2}.\label{eqn:norm_def_2_const}
\end{gather}
\end{subequations}

\end{definition}

\begin{theorem}[Convergence for a given electric field]\label{thm:convergence}
Suppose that $\omega \in L^{\infty}(X)$ and consider the problem \eqref{eqn:var_form_decoupled} with $\mathcal{L} = 0$ and an electric field $E_h$ derived from a given continuous potential: $E_h = \nabla_x \Phi_h$. 
With the operators defined in Section \ref{sec:formulation_of_fixed_point}, define the map $\contractmap:\cG_h^1(X) \times \cG_h^1(X) \rightarrow \cG_h^1(X) \times \cG_h^1(X)$ by
\begin{equation}\label{eqn:NLS_map}
\contractmap(\sigma_h, r_h) := (\intv \mmap(E_h,\sigma_h, r_h),\trace_{E_h} \mmap(E_h,\sigma_h, r_h)).
\end{equation}
Then $R$ satisfies
\begin{equation}\eqnnum
\|\contractmap(\sigma_h,r_h)\|_{_{\emph{NLS}_j}} \leq \kappa_{_{\emph{NLS}_j}} \|(\sigma_h,r_h)\|_{_{\emph{NLS}_j}} + \left(\veps\tau_h(\mmap(E_h,\sigma_h, r_h), \sigma_h, r_h)\right)^{1/2}
\end{equation}
with $j=1$ for any $\varepsilon>0$, and with $j=2$ for $\varepsilon$ sufficiently small.
\end{theorem}
\begin{remark}
We include two different norm and contraction constant pairs because while $(\|\cdot\|_{_{\emph{NLS}_1}},\kappa_{_{\emph{NLS}_1}})$ is valid for all $\varepsilon > 0$, the second pair $(\|\cdot\|_{_{\emph{NLS}_2}},\kappa_{_{\emph{NLS}_2}})$ provides a sharper estimate when $\varepsilon$ is sufficiently small. One sufficient condition for the second pair to be valid is that $\veps < \max(\Delta t^{-2}, \omega_{\min}^{2/5})$. For comparison with the NEST method, we note that a similar and simpler proof (not shown) shows that the outer iteration of the NEST (Algorithm \ref{alg:nest}) converges with the norm and contraction constant pair
\begin{equation}\label{eqn:contraction_constant_prediction}
\|\sigma_h\|_{_\emph{NEST}}^2 :=  \frac{1}{2}\int_X 
   \left( \frac{2\varepsilon^2}{\Delta t} + \omega \right)\sigma^2 dx , \quad \kappa_{_\emph{NEST}} = \bigg( \frac{\omega_{\max}}{2\varepsilon^2\Delta t^{-1}+ \omega_{\max}} \bigg)^{1/2}.
\end{equation}
With everything fixed except for $\veps$, we see that

\begin{equation}
\|(\sigma_h,r_h)\|_{_{\emph{NLS}_2}} \sim \|\sigma_h\|_{_{\emph{NEST}}}, \quad \kappa_{_{\emph{NLS}_2}} \sim \kappa_{_{\emph{NEST}}} \quad \text{as } \varepsilon \rightarrow 0.
\end{equation}
This suggests that NLS should converge in around the same number of outer iterations as NEST for $\varepsilon$ small. Since each outer iteration of NEST requires \emph{at least} one sweep, but an iteration of NLS requires only one sweep, we predict NLS should converge faster in terms of wall time when $\varepsilon$ is small.
We verify this behavior and the estimate of $\kappa_{_{\emph{NLS}_2}}$ numerically in Section \ref{sec:numerical_tests}.
\end{remark}

We require the following Hamiltonian identity to prove Theorem \ref{thm:convergence}.
\begin{lemma}Given the relation $E = \nabla_x \Phi$, we have for any $f$ that
\begin{equation}\label{eqn:hamiltonian_vlasov_identity_clean}
\begin{alignedat}{2}
&f\left(v \cdot \nabla_x \left( \frac{f}{M}e^{-\Phi/\Theta}\right) + E \cdot \nabla_v \left( \frac{f}{M}e^{-\Phi/\Theta}\right) \right) \\
&\qquad\qquad= \frac{1}{2} \left(v \cdot \nabla_x \left( \frac{f^2}{M} e^{-\Phi/\Theta}\right)  + E \cdot \nabla_v \left( \frac{f^2}{M} e^{-\Phi/\Theta}\right) \right).
\end{alignedat}
\end{equation}
\end{lemma}

\begin{proof}[Proof of Theorem \ref{thm:convergence}] Define $u_h := \mmap(E_h,\sigma_h, r_h)$, and $(\upsilon_h, s_h) := \contractmap(\sigma_h, r_h)$.
 In \eqref{eqn:var_form_decoupled}, set $\gone_h=\Pi_h\left(\frac{\uone_h}{M}e^{-\Phi_h/\Theta}\right)$ and $\gtwo_h = \Pi_h\left(\frac{\utwo_h}{M}e^{-\Phi_h/\Theta}\right)$.\\
  Using $ \Pi_h\left(\frac{u_h}{M}e^{-\Phi_h/\Theta}\right)=\frac{u_h}{M}e^{-\Phi_h/\Theta} + \phi_h(u_h)$, we get
\begin{equation}\label{eqn:putting_interpolation_error_in_clean}
\begin{alignedat}{2}
\mathcal{A}_{E_h} \bigg(\uone_h,\frac{\uone_h}{M}e^{-\Phi_h/\Theta} \bigg) + \mathcal{A}_{E_h} \bigg(\utwo_h,\frac{\utwo_h}{M}e^{-\Phi_h/\Theta} \bigg) = \mathcal{S}\bigg( \sigma_h ,\frac{u_h}{M}e^{-\Phi_h/\Theta} \bigg)
\\
 -\mathcal{B}^{\two} \bigg ( \rtwo_h,\frac{\uone_h}{M}e^{-\Phi_h/\Theta} \bigg) -\mathcal{B}^{\one} \bigg( \rone_h,\frac{\utwo_h}{M}e^{-\Phi_h/\Theta} \bigg)
 - \tau_h(u_h, \sigma_h, r_h).
\end{alignedat}
\end{equation}
We have from \eqref{eqn:hamiltonian_vlasov_identity_clean}, the upwind flux definitions, and the continuity of $\Phi_h$ that
\begin{align}\label{eqn:A_bound_clean}
\mathcal{A}_{E_h}&\bigg( \uone_h, \frac{\uone_h}{M}e^{-\Phi_h/\Theta} \bigg) + \mathcal{A}_{E_h}\bigg( \utwo_h, \frac{\utwo_h}{M}e^{-\Phi_h/\Theta} \bigg)
\geq\nonumber\\
&\quad\int_{{\comp}}
\left(\frac{\veps}{\Delta t} + \frac{\omega}{\veps} \right)\frac{u_h^2}{M}e^{-\Phi_h/\Theta} dvdx
+\frac{1}{2}\int_{X}|E_h| \left(u_h^+(x,0)^2 + u_h^-(x,0)^2\right)\frac{e^{-\Phi_h/\Theta}}{M_0}dx.\eqnnum
\end{align}
Now we deal with the source term $\mathcal{S}$, which is defined in \eqref{eqn:operatorQ}. From Cauchy-Schwarz and Young's inequalities,
\begin{equation}\label{eqn:S_bound_clean}
\mathcal{S} \bigg( \sigma_h, \frac{u_h}{M}e^{-\Phi_h/\Theta} \bigg) 
\leq
\frac{1}{2} \int_{X} \frac{\omega}{\varepsilon}  \sigma_h^2 e^{-\Phi_h/\Theta}dx  + \frac{1}{2}\int_{\comp}\frac{\omega}{\varepsilon}  \frac{u_h^2}{M} e^{-\Phi_h/\Theta} dvdx.
\end{equation}
Next, we deal with the coupling terms $\mathcal{B}^{\one}$ and $\mathcal{B}^{\two}$,
\begin{align}
\mathcal{B}^{\one}& \bigg( \rone_h, \frac{\utwo_h}{M}e^{-\Phi_h/\Theta} \bigg) + \mathcal{B}^{\two} \bigg( \rtwo_h, \frac{\uone_h}{M}e^{-\Phi_h/\Theta} \bigg)
\leq \frac{1}{2}  \int_X|E_h|\frac{r_h^2 }{M_{0}}e^{-\Phi_h/\Theta}  dx 
\nonumber\\
&\quad\qquad+ \frac{1}{2}  \int_{E_h < 0}|E_h|\frac{u^{-}_h(x,0)^2 }{M_{0}}e^{-\Phi_h/\Theta}  dx\eqnnum
+\frac{1}{2}  \int_{E_h > 0}|E_h|\frac{u^{+}_h(x,0)^2 }{M_{0}}e^{-\Phi_h/\Theta}dx .\label{eqn:B_bound_clean2}
\end{align}
Using bounds \eqref{eqn:A_bound_clean}--\eqref{eqn:B_bound_clean2} in \eqref{eqn:putting_interpolation_error_in_clean}, and grouping the $\omega u_h$, $u_h^+$, and $u_h^-$ terms,
\begin{align} \label{eqn:before_bounding_below}
\int_{{\comp}}
\bigg(\frac{\varepsilon}{\Delta t} &+ \frac{\omega}{2\veps} \bigg)\frac{u_h^2}{M}e^{-\Phi_h/\Theta} dvdx + \frac{1}{2}\int_{X}|E_h| \frac{(\hat{u}_h)^2}{M_0}e^{-\Phi_h/\Theta}dx
\leq \nonumber\\
&\quad
\frac{1}{2}  \int_X|E_h|\frac{r_h^2 }{M_0}e^{-\Phi_h/\Theta}  dx 
+\frac{1}{2} \int_{X} \frac{\omega}{\varepsilon} \sigma_h^2 e^{-\Phi_h/\Theta}  dx 
+ \tau_h(u_h, \sigma_h, r_h).
\end{align} 
\noindent
We bound the volumetric term on the left-hand side of \eqref{eqn:before_bounding_below} below by first splitting it, and then using both $\upsilon_h^2 \leq \int\frac{u_h^2}{M} dv$ and an inverse inequality,
\begin{align}\label{eqn:splitting_volume_term}
\int_{{\comp}}
\bigg(\frac{\veps}{\Delta t} +& \frac{\omega}{2\veps} \bigg)\frac{u_h^2}{M}e^{-\Phi_h/\Theta} dvdx 
\geq\nonumber\\
&\int_X
   \left(\frac{\veps}{2\Delta t} + \frac{\omega}{2\veps} \right)\upsilon_h^2 e^{-\Phi_h/\Theta} dx 
   +
   \frac{C\Delta v\veps}{2\Delta t}\int_X
    \frac{s_h^2}{M_0}e^{-\Phi_h/\Theta} dx.
\end{align}
Plugging \eqref{eqn:splitting_volume_term} into \eqref{eqn:before_bounding_below}, multiplying the resulting inequality by $\veps$, and then using the elementary inequality 
$
g(x) \leq \frac{\|g\|_{L^{\infty}}}{C + \|g\|_{L^{\infty}}} (C + g(x)) 
$
leads to
\begin{align}
   \frac{1}{2}\int_X
   \bigg(&\frac{\veps^2}{\Delta t} + \omega \bigg)\upsilon_h^2 e^{-\Phi_h/\Theta} dx 
   +
   \frac{\veps}{2}\int_X \left(\frac{C\Delta v \veps}{\Delta t} + |E_h| \right)
    \frac{s_h^2}{M_0}e^{-\Phi_h/\Theta} dx \nonumber\\
\leq \,\,&
\frac{\veps}{2}\left(\frac{\|E_h\|_{L^\infty}}{\sfrac{C\Delta v\veps}{\Delta t} + \|E_h\|_{L^\infty}} \right)\int_X\left(\frac{C\Delta v\veps}{\Delta t} + |E_h| \right)\frac{r_h^2 }{M_0}e^{-\Phi_h/\Theta}  dx \label{eqn:converge_final_bound}
\\
&+\frac{1}{2} \left(\frac{\omega_{\max}}{\sfrac{\veps^2}{\Delta t} + \omega_{\max}} \right) \int_{X} \left(\frac{\veps^2}{\Delta t} + \omega \right)  \sigma_h^2 e^{-\Phi_h/\Theta}dx  
+\varepsilon \tau_h(u_h, \sigma_h, r_h). \nonumber
\end{align}

This concludes the proof for the $(\|\cdot\|_{_{\textup{NLS}_1}},\kappa_{_{\textup{NLS}_1}})$ case. If $\normweight > 0$, we derive the bound for $(\|\cdot\|_{_{\textup{NLS}_2}},\kappa_{_{\textup{NLS}_2}})$ by splitting the left-hand side of \eqref{eqn:splitting_volume_term} differently as
\begin{gather}\eqnnum
   \int_{\comp}
   \left(\frac{\varepsilon}{\Delta t} + \frac{\omega}{2\varepsilon} - \varepsilon^{3/2}\right)\frac{u_h^2}{M}e^{-\Phi_h/\Theta} dvdx 
   +
   \int_{\comp}
   \varepsilon^{3/2} \frac{u_h^2}{M}e^{-\Phi_h/\Theta} dvdx.
\end{gather}
Applying similar inequalities as used in \eqref{eqn:splitting_volume_term} and \eqref{eqn:converge_final_bound} yields the result.
\end{proof}

\section{Acceleration methods\label{sec:solver_strategies}}

In this section, we describe the acceleration methods employed to speed up the solvers considered in this paper. In Section \ref{sec:new_solver_formulation}, we developed and analyzed a new, low-dimensional fixed-point formulation. While standard fixed-point iteration may be slow, we can use \eqref{eqn:fixed_point_iteration} as a framework on which to build faster methods. In the ensuing sections, we outline two different acceleration strategies which may be used independently or together.  The first is Anderson Acceleration (AA), which uses previous residuals of the fixed-point iteration to select the next update. We recall AA in Section \ref{sec:anderson_acceleration}. The second, which is inspired by the diffusion synthetic acceleration (DSA) method used in radiation transport \cite{Adams-Larsen-2002,Alcouffe-1976,Alcouffe-1977}, uses the drift-diffusion equations to accelerate the iterations when $\veps$ is small.  This acceleration strategy is derived in Section \ref{sec:dsa}.

\subsection{Anderson Acceleration\label{sec:anderson_acceleration}}
Consider a generic fixed-point problem, $y = G(y)$ with $G$ defined on some Hilbert space.  If $G$ is a contraction, this equation can be solved with fixed-point iteration, 
$y^{k+1} = G(y^k)$. AA can speed up these fixed-point iterations, and even result in convergence when $G$ is otherwise not contractive \cite{pollock2019anderson}. The algorithm is given below as it appears in \cite{walker2011anderson}.  We set the relaxation parameter $\beta^k = 1$, as is done in the  analysis in \cite{walker2011anderson}.  However, smaller values may be needed to ensure convergence and $\beta^k$ may even be chosen adaptively \cite{evans2018proof}. 

\begin{algorithm}\caption{Anderson Acceleration - AA}
Given $y^0$, $m\geq 1$, and $\beta^k \in (0,1]$,
set $y^1 = G(y^0)$.
For $k = 1,2,...$
Set $m^k = \min(m,k)$.
Set $r^k:=G(y^k) - y^k$.
At iteration $k$, determine $\alpha^{k}$ that solves
\begin{equation}
\min_{\alpha = (\alpha_0,...,\alpha_{m^{k}})} \left\| \sum\nolimits_{i = 0}^{m^k} \alpha_i r^i \right\| \text{ s.t. }
 \sum\nolimits_{i} \alpha_i = 1.
\end{equation}

Set $y^{k+1} = (1-\beta^k)\sum_{i=0}^{m^k} \alpha_i^{k}y^{k-m^k+i} + \beta^k\sum_{i=0}^{m^k} \alpha_i^{k}G(y^{k-m^k+i})$. 
\end{algorithm}

Similar to the generalized minimal residual method (GMRES) \cite{saad1986gmres} for linear systems, an important practical aspect of AA is the choice of $m$ based on memory constraints and the condition number of the resulting least squares problem, which scales poorly with $m$ \cite{ni2010linearly}.  The new fixed-point formulation \eqref{eqn:fixed_point_problem} is helpful here since it is posed on a lower dimensional space that enables the storage of more solution vectors.   In Section \ref{sec:numerical_tests}, we investigate numerically the performance of AA applied to \eqref{eqn:fixed_point_problem}.
\subsection{The drift-diffusion accelerator\label{sec:dsa}}

The Picard iteration in Algorithms \ref{alg:nest} and \ref{alg:nls} becomes less effective as $\veps$ gets small, so at some point, acceleration or preconditioning becomes necessary. Even in the simple case of a fixed electric field, the contraction constants for NLS \eqref{eqn:norm_def_1_const},\eqref{eqn:norm_def_2_const} and NEST \eqref{eqn:contraction_constant_prediction} tend to one as  $\veps \to 0$.  To address this problem, we  implement an acceleration strategy which relies on two key facts.  First is that the solution of the simplified Boltzmann-Poisson system \eqref{eqn:model} approaches the solution of a drift-diffusion-Poisson system as $\veps \to 0$ \cite{abdallah2004, masmoudi2007diffusion}. Second, because  it is independent of velocity, the drift-diffusion-Poisson system is much cheaper to solve than the simplified Boltzmann-Poisson system \eqref{eqn:model}.  

The drift-diffusion Poisson system takes the form
\begin{subequations}\label{eqn:drift-diffusion-equation}
\begin{align}\label{eqn:drift-diffusion-equation-rho}
&\partial_t \rho_0 - \nabla_x \cdot (\omega^{-1} \nabla_x \rho_0) + \nabla_x \cdot(\omega^{-1}E\rho_0) = \int q dv, \quad \rho_0|_{\partial X} = \int f_- dv,\\
&\qquad \, E_0 = \nabla_x \Phi_0, \qquad  \Delta_x\Phi_0 =\rho_0 - D,  \quad  \Phi_{0}|_{\partial X} = \potbound.
\end{align}
\end{subequations}
In radiation transport, the use of diffusion equations to accelerate  iterative methods for their kinetic antecedents is referred to as diffusion synthetic acceleration (DSA) \cite{Adams-Larsen-2002}.  We borrow from the nomemclature and refer to  the use of accelerators based on \eqref{eqn:drift-diffusion-equation} as drift-diffusion synthetic acceleration (DDSA).

A DDSA correction for the NEST algorithm (Algorithm \ref{alg:nest}) was derived in \cite{laiu2019fast} to correct the $\rho$ iterate. We seek a similar DDSA correction for the NLS algorithm (Algorithm \ref{alg:nls}), which has the additional complication that $\widetilde{f}$ and $\rho$ are iterated simultaneously, so we must derive a correction for $\tilde{f}$ as well. However, we can use the fact that $f_{\veps} \xrightarrow{\varepsilon \rightarrow 0} M\rho_0$ to construct a low-order approximation $\tilde{f}_\veps \approx M(0) \rho_0$, where $\rho_0$ is the solution to \eqref{eqn:drift-diffusion-equation}.

To simplify the discussion, we formally derive the drift-diffusion synthetic accelerator using a one-dimensional steady-state form of \eqref{eqn:time_discrete}  with the drift-diffusion scaling, under the assumptions that the solution $f_{\veps}$ is sufficiently smooth and the electric field $E$ is fixed and $\veps$-independent%
\footnote{In practice, we still allow $E$ to change in  each iteration in the implementation.}:
\begin{gather}\label{eqn:dd-continuous}
\frac{\omega_\ast}{\veps}f_{\veps} + v \partial_x f_{\veps} + E \partial_v f_{\veps}= \veps q_\ast + \frac{\omega}{\veps} M \mathcal \rho_{\veps},
\end{gather}
where $q_\ast = \Delta t^{-1}f_{\veps}^{n} +q^{n+1}$ incorporates previous time step information and $\omega_\ast = \omega  + \varepsilon^2\Delta t^{-1}$.
With the notation introduced in \eqref{eqn:velocity_decomp} and \eqref{eq:subdomain_traces_cont}, let $\fonetilde_{\veps} := \trace^{\one}_{E} \fone_{\veps}$, and $\ftwotilde_{\veps} := \trace^{\two}_{E} \ftwo_{\veps}$. Then \eqref{eqn:dd-continuous} is equivalent to
\begin{subequations}\label{eqn:dd-continuous-dd} 
\begin{align}
\frac{\omega_\ast}{\veps}\fone_{\veps} + v \partial_x \fone_{\veps} + E \partial_v \fone_{\veps}= \veps \qone_\ast +\frac{\omega}{\veps} \Mone  \rho_{\veps},
\quad
\fone_{\veps} \big| _{B^{\one}} = \ftwotilde_{\veps}, \\
\frac{\omega_\ast}{\veps}\ftwo_{\veps} + v \partial_x \ftwo_{\veps} + E \partial_v \ftwo_{\veps}= \veps \qtwo_\ast + \frac{\omega}{\veps} \Mtwo  \rho_{\veps},
\quad
\ftwo_{\veps} \big| _{B^{\two}} = \fonetilde_{\veps},
\end{align}
\end{subequations}
where $B^{\one}= \{(x,v) \in Z : E>0, v=0 \}$ and $B^{\two}= \{(x,v) \in Z : E<0, v=0 \}$.
When applied to \eqref{eqn:dd-continuous-dd}, the $k+1$-th iterate of the NLS algorithm takes the form $(\rho_{\veps}^{k_\ast},\widetilde{f}_{\veps}^{{k_\ast}}) = (P f_{\veps}^{k_*},\,\trace_E f_{\veps}^{k_*})$, with $f_{\veps}^{k_*}$ satisfies
\begin{subequations}\label{eqn:dd-continuous-dd-NLS} 
\begin{align}
\frac{\omega_\ast}{\veps}\fone[k_\ast]_{\veps} + v \partial_x \fone[k_\ast]_{\veps} + E \partial_v \fone[k_\ast]_{\veps}
	&= \veps \qone_\ast +\frac{\omega}{\veps} \Mone \mathcal  \rho_{\veps}^ {k},
\,\,
\fone[k_\ast]_{\veps} \big| _{B^{\one}}  = \ftwotilde[k]_{\veps}, \\
\frac{\omega_\ast}{\veps}\ftwo[k_\ast]_{\veps}+ v \partial_x \ftwo[k_\ast]_{\veps} + E \partial_v \ftwo[k_\ast]_{\veps}
	&= \veps \qtwo_\ast + \frac{\omega}{\veps} \Mtwo  \rho_{\veps}^{k},
\,\,
\ftwo[k_\ast]_{\veps} \big| _{B^{\two}}  = \fonetilde[k]_{\veps}. 
\end{align}
\end{subequations}
Let $\psi_{\veps} = f_{\veps} - f_{\veps}^{k_\ast}$ and $\phi_{\veps} = \intv_{\veps} \psi_\veps$. Clearly, if we know $\phi_\veps$, then we would not need to iterate since $\rho_\veps = \rho_\veps^{k_*} + \phi_\veps$. The goal is to obtain a low order approximation to $\phi_\veps$. Subtracting \eqref{eqn:dd-continuous-dd-NLS}  from \eqref{eqn:dd-continuous-dd} and applying $\intv$ to the resulting equation gives
\begin{gather}\label{eqn:phi_eqn}
\frac{\veps}{\Delta t} \phi_{\veps} + \partial_x \left(\int v  \psi_{\veps} dv \right) = |E|(\widetilde{f}_{\veps}^{{k_\ast}} - \widetilde{f}_{\veps}^k) + \frac{\omega}{\veps}(\rho_{\veps}^{{k_\ast}} - \rho_{\veps}^{k}).
\end{gather} 
We approximate $\phi_{\veps}$ by considering the right-hand side of \eqref{eqn:phi_eqn} as a source, and formally taking $\veps\rightarrow 0$ on the left-hand side which gives the one dimensional implicit Euler discretization of the drift-diffusion equation \eqref{eqn:drift-diffusion-equation-rho}, 
\begin{equation}\label{eqn:drift_diffusion_operator}
D_E \phi_0 := \veps \left(\frac{1}{\Delta t}\phi_0 - \partial_x(\omega^{-1}\partial_x \phi_0) + \partial_x(\omega^{-1}E\phi_0)   \right) = |E|(\widetilde{f}_{\veps}^{{k_\ast}} - \widetilde{f}_{\veps}^k) + \frac{\omega}{\veps}(\rho_{\veps}^{{k_\ast}} - \rho_{\veps}^{k}),
\end{equation}
together with the boundary condition $\phi_0|_{\partial X} = 0$.
The correction to the iterate is $\rho_{\veps}^{k+1} = \rho_{\veps}^{{k_\ast}} + \phi_0$.
To correct the numerical trace values, we approximate $\psi_{\veps}$ by $M\phi_0$ since $\psi_{\veps}\to M\phi_0$ in the drift-diffusion limit. The update is $\widetilde{f}_{\veps}^{k+1} = \widetilde{f}_{\veps}^{{k_\ast}} + M(0)\phi_0$.

The drift-diffusion operator $D_E$ in \eqref{eqn:drift_diffusion_operator} is discretized using the direct discontinous Galerkin method with interface correction (DDG-IC), first developed in \cite{Liu2010} for convection-diffusion problems.  Details of the fully discrete algorithm for the drift-diffusion equations (used herein) can be found in \cite{laiu2019fast}. For the nonlinear iterations, the electric field changes each iteration so the drift-diffusion operator is actually $D_{F(\rho_h^k)}$.
 One drawback of using this discretization is that it may not be equal to the limiting kinetic discretization as $\veps \rightarrow 0$ (although the limiting discretization is itself a valid discretization of the drift-diffusion limit).  While our tests do demonstrate accelerated convergence when $\veps$ is small, the accelerator may destabilize the Picard iteration if $\Delta t$ is too large, even for a fixed $E$.  In the context of radiation transport, accelerated methods that preserve the stability of Picard iteration have been derived, and work robustly across a range of discretization parameters \cite{Alcouffe-1976, Alcouffe-1977}. We leave the use of such accelerators for this problem for future work.

\section{Numerical Tests\label{sec:numerical_tests}}
In this section, we compare the solvers based on the new fixed point formulation (Algorithm \ref{alg:nls}) with the nested iterative formulation developed in \cite{laiu2019fast} (Algorithm \ref{alg:nest}). We compare the total number of sweeps required to run a simulation to completion since the number of sweeps is directly related to the computational effort. We also report the total runtime.

Because the new fixed point formulation is on a lower dimensional space than the phase space, we may effectively employ Anderson Acceleration with much lower spatial complexity than if it were formulated on the entire phase space. We test the methods with and without Anderson Acceleration, and with and without drift-diffusion synthetic acceleration (DDSA).

 In the plots and tables that follow, the new nonlinear sweeping algorithm (Algorithm \ref{alg:nls}) will be denoted by NLS, and the nested algorithm  (Algorithm \ref{alg:nest}) by NEST. The methods with Anderson Acceleration are followed by `AA', and those without Anderson Acceleration are followed by `PIC', short for `Picard'. If DDSA is used in conjunction with any of the solvers, `+DDSA' is appended to the end of the name. For example, NLS-PIC+DDSA means the NLS method with Picard iteration and drift-diffusion synthetic acceleration.

\subsection{Problem setting\label{sec:diode_details}}
The tests in Sections~\ref{sec:test_single_scale}--\ref{sec:test_iter_rate} consider a one-dimensional diode with several variations of the collision frequency.
We recall the scaled model from \cite{laiu2019fast} which is derived from a nondimensionalization of the simplified Boltzmann-Poisson system: 
\begin{subequations}
\begin{align}\label{eqn:nondim_model}
&\varepsilon \partial_t f + v \partial_x f + \beta^2 E \partial_v f = \frac{\omega}{\varepsilon}(M_{\alpha^2} \rho - f),\\
E = \partial_x &\Phi, \quad \partial_x^2 \Phi = \frac{\gamma^2}{\beta^2}(\zeta \rho - D), \quad f|_{\partial X_-} = DM_{\alpha^2}|_{\partial X_-}.
\end{align}
\end{subequations}
The parameters $\alpha$, $\beta$, $\gamma$, and $\zeta$ are nondimensional quantities. After setting the physical quantities from which they are derived to those used in \cite[Section 4.1]{laiu2019fast}, the nondimensionial quantities are $\alpha=0.129$, $\beta = 0.803$, $\gamma = 1$, and $\zeta = 1$.

The nondimensionalized spatial domain is $X = [0,0.6]$. The boundary conditions for the Poisson problem are set to $\Phi(0) = 0$ and $\Phi(0.6) = 1$.
The remaining parameter $\varepsilon$ changes depending on the test.
The same doping profile $D(x)$ used in \cite{laiu2019fast}, after non-dimensionalization, is 500 at the boundaries with a smooth, but sharp, transition to a value of 2 in between the boundaries.

In Section~\ref{sec:convergence_rate}, numerical tests are performed on problems with a manufactured solution. The problem setting and implementation details are described therein.

\subsection{Discretization and solver details\label{sec:test_disc_details}}
 The computational domain for the phase space, $\comp \subset Z$, is $\comp = [0,L]\times[-v_{\max},v_{\max}] = [0,0.6]\times[-2,2]$, which corresponds to a diode of length $0.6 \mu m$. The velocity space is truncated so that the tail of $M_{\alpha^2}$ is below machine precision outside the computational domain.
The initial condition is always set according to the doping profile:  $f_h|_{t=0}(x,v) = \Pi_h D(x)M_{\alpha^2}(v)$. 
The incoming data at the artificial boundaries $v = \pm 2$ are set to zero.

The iterative solvers are always initialized with the solution from the previous time step or the initial condition in the case of the first time step. The Anderson Acceleration restart parameter is set at $m=15$ for both the NLS and NEST methods to ensure they have similar memory complexity. The restart for NEST-AA was $m=3$ in \cite{laiu2019fast}, but the memory footprint of NEST is potentially larger due to the inner GMRES solve. We have observed inner GMRES solves in NEST taking up to 15 iterations, which is the reason for this choice. The least squares problem arising from Anderson acceleration is solved via QR decomposition using the Eigen library \cite{eigenweb}.

In Sections \ref{sec:test_single_scale} and \ref{sec:test_silicone_diode}, we report the total execution time and the total number of sweeps to finish the entire simulation for each problem configuration with various $\Delta t$. Most tests in these sections are performed on the same uniform rectangular mesh with $200^2$ elements, resulting in 120,000 degrees of freedom in the phase space. The final time $T_f$ is always set to $0.5$ since the solutions are near the steady-state by then. For each test, we include two separate sections of the table for the methods with and without DDSA. We explicitly mention any deviation from these default mesh parameters when they occur.

The iterations are stopped when the relative  $\ell^2$ residual is less than a specified tolerance. The $\ell^2$ norm refers to the norm of the vector of coefficients for the DG representations of either $\rho_h$ for NEST, or $(\rho_h, \widetilde{f}_h)$ for NLS.  For NLS, the tolerance for this norm is set to $10^{-8}$. NEST requires two tolerances, one for each level of iteration.  As in \cite{laiu2019fast}, the tolerance for the outer loop is set to $10^{-8}$, and the relative tolerance for the inner loop is set to $10^{-10}$.  It is likely the NEST could be made more efficient by using an adaptive strategy for the inner sweeping iterations, such as starting with a larger relative tolerance for the inner loop, and decreasing as needed. This possibility is not explored herein.

 In practice, iterative methods which converge very slowly should have a modified tolerance since they can exhibit so called ``false convergence'' \cite{Adams-Larsen-2002}. If the stopping criteria is that the norm of the difference of two iterates be less than $\eta$, and if $\kappa$ is the contraction constant, then the actual error between the last iterate and the exact fixed point may be as large as $\frac{\kappa \eta}{1 - \kappa}$. Therefore, when $\kappa$ is close to one, the stopping criteria should be scaled by $1-\kappa$. An analytic estimate for $\kappa$ for both NEST and NLS is given by \eqref{eqn:contraction_constant_prediction} when $\varepsilon$ is small. While we do not have an estimate for $\kappa$ when DDSA is used, it should have the effect of reducing $\kappa$ when $\veps$ is small, which both speeds up convergence and results in less accuracy loss due to slow convergence.

For consistency across the tests, we use a static tolerance that does not take $\kappa$ into account. This gives the methods without DDSA an advantage when $\varepsilon$ is small since the true error may actually be $\kappa(1-\kappa)^{-1}$ times larger than the final residual. Even without this advantage, DDSA still results in faster convergence for small $\varepsilon$. We note that even in the worst case we test when $\varepsilon = 0.002$ and $\Delta t = 0.25$, the analytic estimates on $\kappa$ suggest that the methods without DDSA lose at most five significant digits of accuracy. Since the relative tolerance is set to $10^{-8}$, the converged solution without DDSA should still have several significant digits of accuracy.

For all methods, we limit the total number of sweeps per time step to 50,000. One sweep for NEST is defined to be the solution of \eqref{eqn:nested_sweep_1} and \eqref{eqn:nested_sweep_2}, and one sweep for NLS is the solution of \eqref{eqn:nls-explicit_1} and \eqref{eqn:nls-explicit_2}. If a method fails to converge within this limit or if the residuals blow up, the simulation is terminated. We report three types of non-convergence.

\begin{enumerate}
\item INF - Divergence to infinity. The iterations are unstable and the residuals diverge to infinity.
\item ($r$) - Did not converge, with a final relative residual of $r$.
\item FC - Falsely converged. Since a successfully converged solution for these problems has a squared $L^2$ norm of around $10^5$, we say any solution with a squared $L^2$ norm greater than $2\times10^5$ or below $5\times10^4$ falsely converged.
\end{enumerate}

\subsection{Single-scale test\label{sec:test_single_scale}}

In this section, we test the single scale case where the collision frequency $\omega$ does not vary in space. In this section, $\omega = 1$ in the entire domain and $\varepsilon = 0.2$. 

Efficiency  results are reported in Table \ref{tab:single_scale0.2}. The fastest methods for a given timestep were the NLS based methods except for the largest timestep, where NEST-AA+DDSA was the fastest. However, DDSA does not yield any significant benefit, since the solution is far from the drift-diffusion regime.

\begin{table}[h]
\centering
\begin{tabular}{|l|ll|ll|ll|ll|}
\hline
  &   \multicolumn{2}{|c|}{NLS-AA} &   \multicolumn{2}{|c|}{NEST-AA} &   \multicolumn{2}{|c|}{NLS-PIC} &   \multicolumn{2}{|c|}{NEST-PIC}   \\
$\Delta t$ & time(s)  &   swps.    & time(s)  &   swps.    & time(s)  &   swps.    & time(s)  &   swps.     \\
\hline
\multicolumn{9}{|c|}{Without DDSA}\\
\hline
$T_f/2^{1}$ &    3.98 &    643  &    4.02 &    894 & \multicolumn{2}{c|}{R(8.4E-1)}& \multicolumn{2}{c|}{R(6.7E-1)} \\
$T_f/2^{2}$ &    1.53 &    274  &    2.72 &    596 & \multicolumn{2}{c|}{R(7.2E-1)}& \multicolumn{2}{c|}{R(7.0E-1)} \\
$T_f/2^{3}$ &    2.02 &    330  &    4.06 &    865 & \multicolumn{2}{c|}{R(7.6E-1)}& \multicolumn{2}{c|}{R(7.7E-1)} \\
$T_f/2^{4}$ &    2.37 &    293  &    5.45 &   1214 & \multicolumn{2}{c|}{R(8.9E-1)}& \multicolumn{2}{c|}{R(6.3E-1)} \\
$T_f/2^{5}$ &    2.39 &    380  &    7.24 &   1530 & \multicolumn{2}{c|}{R(1.6E-1)}& \multicolumn{2}{c|}{R(1.8E-1)} \\
$T_f/2^{6}$ &    2.64 &    493  &    8.51 &   1894  &    2.67 &    633  &   11.11 &   2494  \\
$T_f/2^{7}$ &    4.33 &    732  &   12.31 &   2591  &    3.95 &    897  &   14.59 &   3271  \\
$T_f/2^{8}$ &    6.89 &   1215  &   18.66 &   4199  &    6.46 &   1470  &   24.93 &   5366  \\
\hline
\multicolumn{9}{|c|}{With DDSA}\\
\hline
$T_f/2^{1}$ &    7.22 &   1116  &    2.11 &    451 & \multicolumn{2}{c|}{INF}& \multicolumn{2}{c|}{R(8.3E-1)} \\
$T_f/2^{2}$ &    8.35 &   1209  &    3.02 &    597 & \multicolumn{2}{c|}{INF}& \multicolumn{2}{c|}{R(8.5E-1)} \\
$T_f/2^{3}$ &    3.12 &    345  &    5.07 &    893 & \multicolumn{2}{c|}{INF}& \multicolumn{2}{c|}{R(8.2E-1)} \\
$T_f/2^{4}$ &    1.93 &    325  &    5.34 &   1237 & \multicolumn{2}{c|}{INF}& \multicolumn{2}{c|}{R(7.7E-1)} \\
$T_f/2^{5}$ &    2.29 &    386  &    6.96 &   1580 & \multicolumn{2}{c|}{R(3.1E-1)}& \multicolumn{2}{c|}{R(3.9E-1)} \\
$T_f/2^{6}$ &    3.37 &    530  &    9.31 &   2077  &    4.79 &    783  &   11.84 &   3287  \\
$T_f/2^{7}$ &    5.10 &    715  &   16.82 &   2761  &    5.09 &    847  &   11.67 &   3246  \\
$T_f/2^{8}$ &    6.94 &   1151  &   23.47 &   4245  &    8.49 &   1420  &   23.90 &   5199  \\
\hline
\end{tabular}
\caption{\textbf{Solver performance on the single scale problem with $\varepsilon = 0.2$.} For the largest timestep ($\Delta t = T_f/2^{1}$), NEST-AA+DSA was the fastest method in terms of number of sweeps and time. For all other $\Delta t$s, one of the NLS based methods was the fastest. The addition of DDSA does not appear to significantly affect the efficiency for this test.\label{tab:single_scale0.2}}
\end{table}

In this next test, $\varepsilon = 0.002$, and the tests are repeated. The results are shown in Table \ref{tab:tab:single_scale0.002}. For a fixed $\Delta t$, NLS-AA+DDSA is the fastest method with the exception of the smallest $\Delta t$, where NLS-AA is the fastest.

\begin{table}[h]
\centering
\begin{tabular}{|l|ll|ll|ll|ll|}
\hline
  &   \multicolumn{2}{|c|}{NLS-AA} &   \multicolumn{2}{|c|}{NEST-AA} &   \multicolumn{2}{|c|}{NLS-PIC} &   \multicolumn{2}{|c|}{NEST-PIC}   \\
$\Delta t$ & time(s)  &   swps.    & time(s)  &   swps.    & time(s)  &   swps.    & time(s)  &   swps.     \\
\hline
\multicolumn{9}{|c|}{Without DDSA}\\
\hline
$T_f/2^{1}$ &   37.91 &  12908  &   22.71 &   8770 & \multicolumn{2}{c|}{R(4.2E-8)}& \multicolumn{2}{c|}{R(1.4E-6)} \\
$T_f/2^{2}$ &   28.40 &   9705  &   23.18 &   8991 & \multicolumn{2}{c|}{R(2.3E-8)}& \multicolumn{2}{c|}{R(1.2E-6)} \\
$T_f/2^{3}$ &   24.30 &   8173  &   31.61 &  12262  &  517.01 & 158905 & \multicolumn{2}{c|}{R(8.6E-7)} \\
$T_f/2^{4}$ &   18.67 &   6357  &   32.02 &  12466  &  546.03 & 182662 & \multicolumn{2}{c|}{R(4.5E-7)} \\
$T_f/2^{5}$ &   14.95 &   4942  &   41.46 &  16236  &  544.54 & 207449 & \multicolumn{2}{c|}{R(1.1E-7)} \\
$T_f/2^{6}$ &   11.09 &   3762  &   37.63 &  13441  &  720.03 & 224085  & 2895.20 & 891543  \\
$T_f/2^{7}$ &   10.02 &   3467  &   32.10 &  11644  &  708.36 & 236401  & 2990.12 & 939791  \\
$T_f/2^{8}$ &   10.78 &   3713  &   29.54 &  11846  &  654.69 & 245274  & 3224.05 & 975130  \\
\hline
\multicolumn{9}{|c|}{With DDSA}\\
\hline
$T_f/2^{1}$ &    0.61 &    160  &    2.69 &    995 & \multicolumn{2}{c|}{INF}& \multicolumn{2}{c|}{FC} \\
$T_f/2^{2}$ &    0.95 &    270  &    4.24 &   1614 & \multicolumn{2}{c|}{INF}& \multicolumn{2}{c|}{FC} \\
$T_f/2^{3}$ &    1.43 &    419  &    6.34 &   2389 & \multicolumn{2}{c|}{INF}& \multicolumn{2}{c|}{FC} \\
$T_f/2^{4}$ &    2.16 &    637  &    9.80 &   3691 & \multicolumn{2}{c|}{INF}& \multicolumn{2}{c|}{R(8.4E-1)} \\
$T_f/2^{5}$ &    3.68 &    958  &   13.94 &   5126 & \multicolumn{2}{c|}{INF}& \multicolumn{2}{c|}{R(8.1E-1)} \\
$T_f/2^{6}$ &    4.88 &   1510  &   19.91 &   7605 & \multicolumn{2}{c|}{INF}& \multicolumn{2}{c|}{R(6.9E-1)} \\
$T_f/2^{7}$ &    7.91 &   2460  &   29.84 &  11323 & \multicolumn{2}{c|}{R(1.5E+0)} &  302.20 &  85437  \\
$T_f/2^{8}$ &   14.49 &   4161  &   44.81 &  17083  &   64.49 &  21488  &  353.31 & 115792  \\
\hline
\end{tabular}
\caption{\textbf{Solver performance on the single scale problem with $\varepsilon = 0.002$.}\label{tab:tab:single_scale0.002}
For the largest timestep $\Delta t = T_f/2^{1}$, NLS-AA+DDSA  is around 50 times faster than NLS-AA. DDSA appears to improve computation time less for smaller $\Delta t$.
The Picard versions of both algorithms perform poorly with and without DDSA. The unaccelerated algorithms eventually converge, but at a rate several orders of magnitude slower than the fastest methods. DDSA can destabilizes the Picard methods except for sufficiently small $\Delta t$.}
\end{table}

\subsection{Silicon diode benchmark problem\label{sec:test_silicone_diode}}

For the tests in this section, $\omega$ varies in space. The first test in this section is a standard benchmark problem of a silicone diode, which was performed in, e.g., \cite{cercignani2000,hu2014}. As in \cite{laiu2019fast}, the resulting dimensionless collision frequency is
\begin{equation}\label{eqn:omega_def}
\omega(x) = 
\begin{cases}
\omega_{\min}, \hspace{5mm} x \in [0.1, 0.5],\\
1, \hspace{10mm} \text{otherwise},
\end{cases}
\end{equation}
with $\omega_{\min} = 0.277$ and $\varepsilon = 0.056$. As in \cite{hu2014,laiu2019fast}, $\omega$ is smoothed out using the same cubic spline transitions from 1 to $\omega_{\min}$.

Efficiency results for the different solvers are shown in Table \ref{tab:silicone}. NLS-AA is consistently the most efficient method across all values of $\Delta t$, in terms of runtime and total sweeps. 
\begin{table}[h]
	\centering
	\begin{tabular}{|l|ll|ll|ll|ll|}
		\hline
		&   \multicolumn{2}{|c|}{NLS-AA} &   \multicolumn{2}{|c|}{NEST-AA} &   \multicolumn{2}{|c|}{NLS-PIC} &   \multicolumn{2}{|c|}{NEST-PIC}   \\
		$\Delta t$ & time(s)  &   swps.    & time(s)  &   swps.    & time(s)  &   swps.    & time(s)  &   swps.     \\
		\hline
		\multicolumn{9}{|c|}{Without DDSA}\\
		\hline
		$T_f/2^{1}$ &    0.32 &     75  &    0.57 &    202  &    0.33 &     87 & \multicolumn{2}{c|}{R(1.4E-1)} \\
		$T_f/2^{2}$ &    0.34 &     97  &    0.85 &    323  &    0.44 &    138 & \multicolumn{2}{c|}{R(1.3E-1)} \\
		$T_f/2^{3}$ &    0.45 &    136  &    1.23 &    480  &    0.60 &    200 & \multicolumn{2}{c|}{R(1.0E-1)} \\
		$T_f/2^{4}$ &    0.54 &    171  &    1.69 &    659  &    0.88 &    299 & \multicolumn{2}{c|}{R(3.5E-2)} \\
		$T_f/2^{5}$ &    0.94 &    258  &    2.33 &    914  &    1.07 &    367  &    5.58 &   1891  \\
		$T_f/2^{6}$ &    1.24 &    361  &    3.57 &   1404  &    1.41 &    478  &    4.78 &   1907  \\
		$T_f/2^{7}$ &    1.74 &    577  &    5.50 &   2195  &    2.12 &    737  &    6.94 &   2781  \\
		$T_f/2^{8}$ &    3.22 &    953  &    8.89 &   3549  &    3.44 &   1182  &   10.91 &   4382  \\
		\hline
		\multicolumn{9}{|c|}{With DDSA}\\
		\hline
		$T_f/2^{1}$ &    0.38 &     84  &    0.62 &    209 & \multicolumn{2}{c|}{INF}& \multicolumn{2}{c|}{R(1.2E+0)} \\
		$T_f/2^{2}$ &    0.46 &    124  &    0.94 &    344 & \multicolumn{2}{c|}{INF}& \multicolumn{2}{c|}{R(7.5E-1)} \\
		$T_f/2^{3}$ &    0.62 &    176  &    1.39 &    517 & \multicolumn{2}{c|}{INF}& \multicolumn{2}{c|}{R(1.1E+0)} \\
		$T_f/2^{4}$ &    0.74 &    215  &    2.59 &    737 & \multicolumn{2}{c|}{INF}& \multicolumn{2}{c|}{R(7.0E-1)} \\
		$T_f/2^{5}$ &    0.97 &    286  &    2.70 &   1038 & \multicolumn{2}{c|}{INF}& \multicolumn{2}{c|}{R(8.3E-1)} \\
		$T_f/2^{6}$ &    1.29 &    381  &    4.79 &   1550 & \multicolumn{2}{c|}{R(9.5E-1)}& \multicolumn{2}{c|}{R(4.3E-1)} \\
		$T_f/2^{7}$ &    1.95 &    584  &    5.86 &   2305  &    4.96 &   1579  &   14.76 &   5705  \\
		$T_f/2^{8}$ &    3.08 &    936  &   10.13 &   3572  &    3.75 &   1142  &   11.04 &   4218  \\
		\hline
	\end{tabular}
	\caption{\textbf{Solver performance on the standard silicone diode problem with $\varepsilon = 0.056$ and $\omega$ from \eqref{eqn:omega_def} with $\omega_{\min} = 0.277$}. \label{tab:silicone} The most efficient solver was NLS-AA, followed closely by NLS-PIC. DDSA did not improve convergence speed for any of the methods, and even causes divergence in some cases.}
\end{table}

To illustrate convergence of the solution under mesh refinement, we plot the electron density $f$ in Figure \ref{fig:final_times}.  For these plots, the underlying computation uses implicit Euler time stepping and is solved with NLS-AA.  For this particular test, we temporarily deviate from the discretization parameters used for the efficiency tests as follows. We shorten the final time to $T_f = 0.05$ so that the solution is not near steady-state. The reference solution shown in Figure \ref{fig:final_times} is calculated with $(\Delta t, \Delta x, \Delta v) = 2^{-10}(T_f, L, 2 v_{\max})$, resulting in 3,145,728 degrees of freedom. We observe that small oscillations in the profile disappear as the mesh is refined.

\begin{figure}[h]
	\centering
	\begin{subfigure}{.32\textwidth}
		\includegraphics[width = \linewidth]{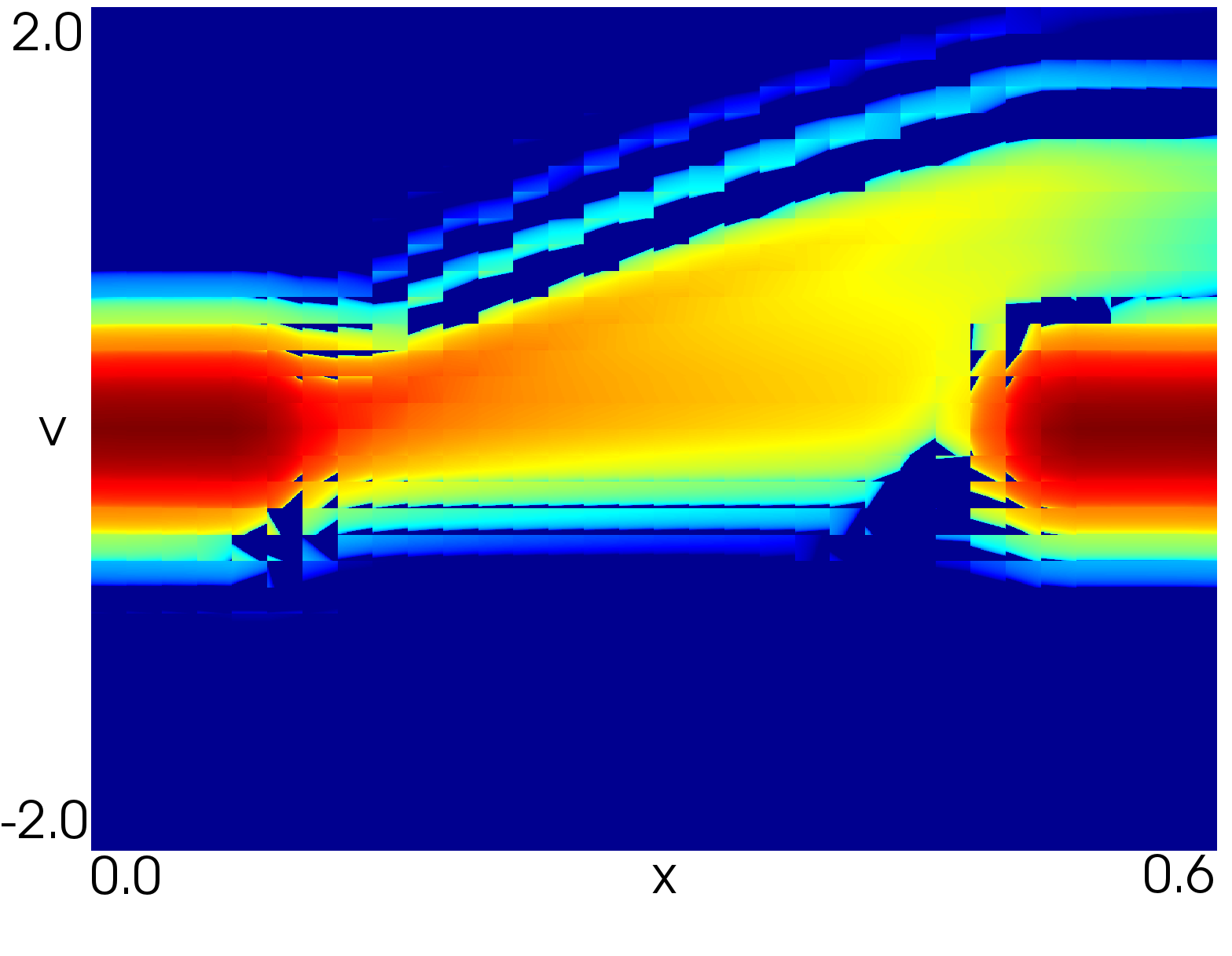}
		\caption*{Level 5}
	\end{subfigure}
	\begin{subfigure}{.32\textwidth}
		\includegraphics[width=\linewidth]{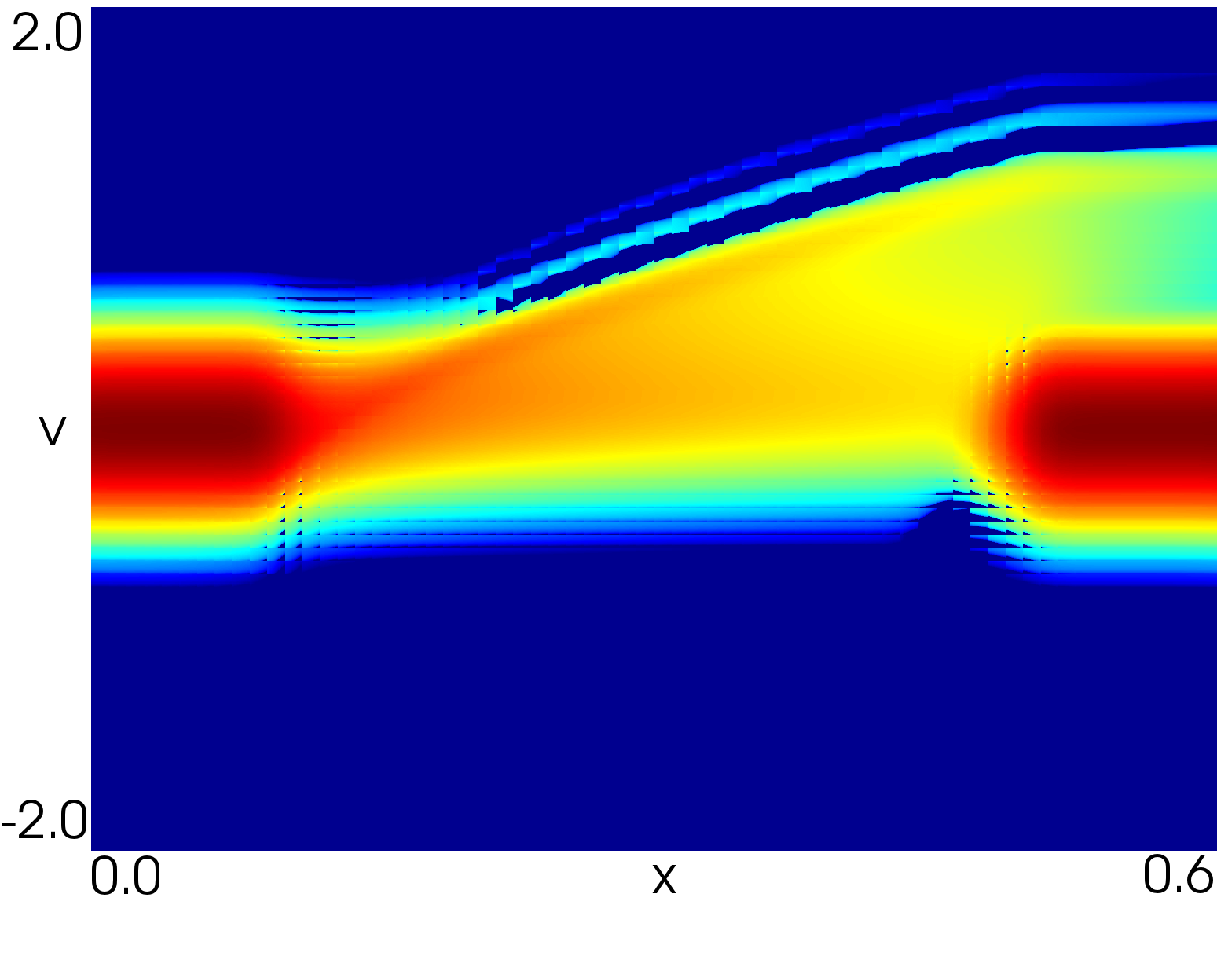}
		\caption*{Level 6}
	\end{subfigure}
	\begin{subfigure}{.32\textwidth}
		\includegraphics[width = \linewidth]{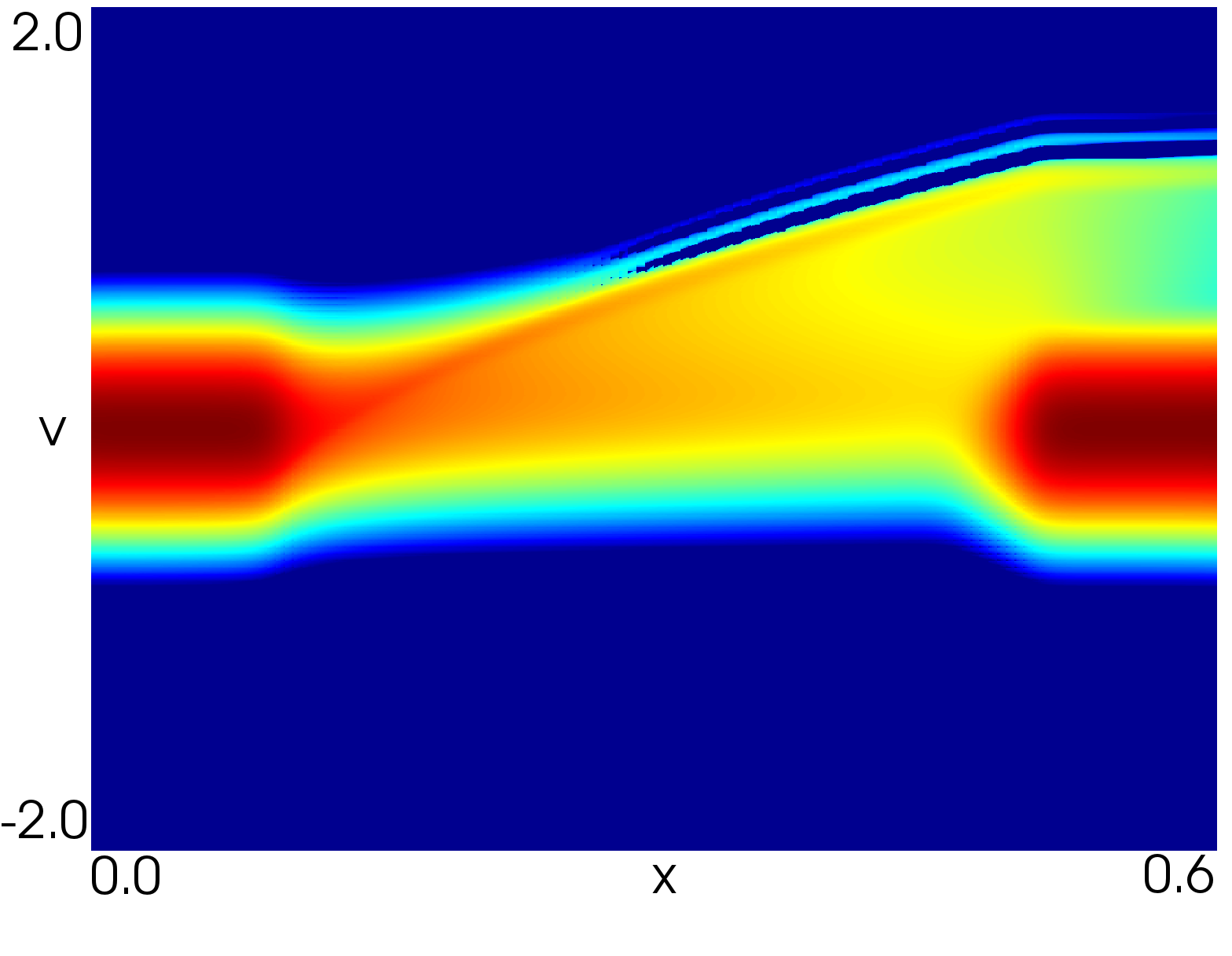}
		\caption*{Level 7}
	\end{subfigure}
	\begin{subfigure}{.32\textwidth}
		\includegraphics[width=\linewidth]{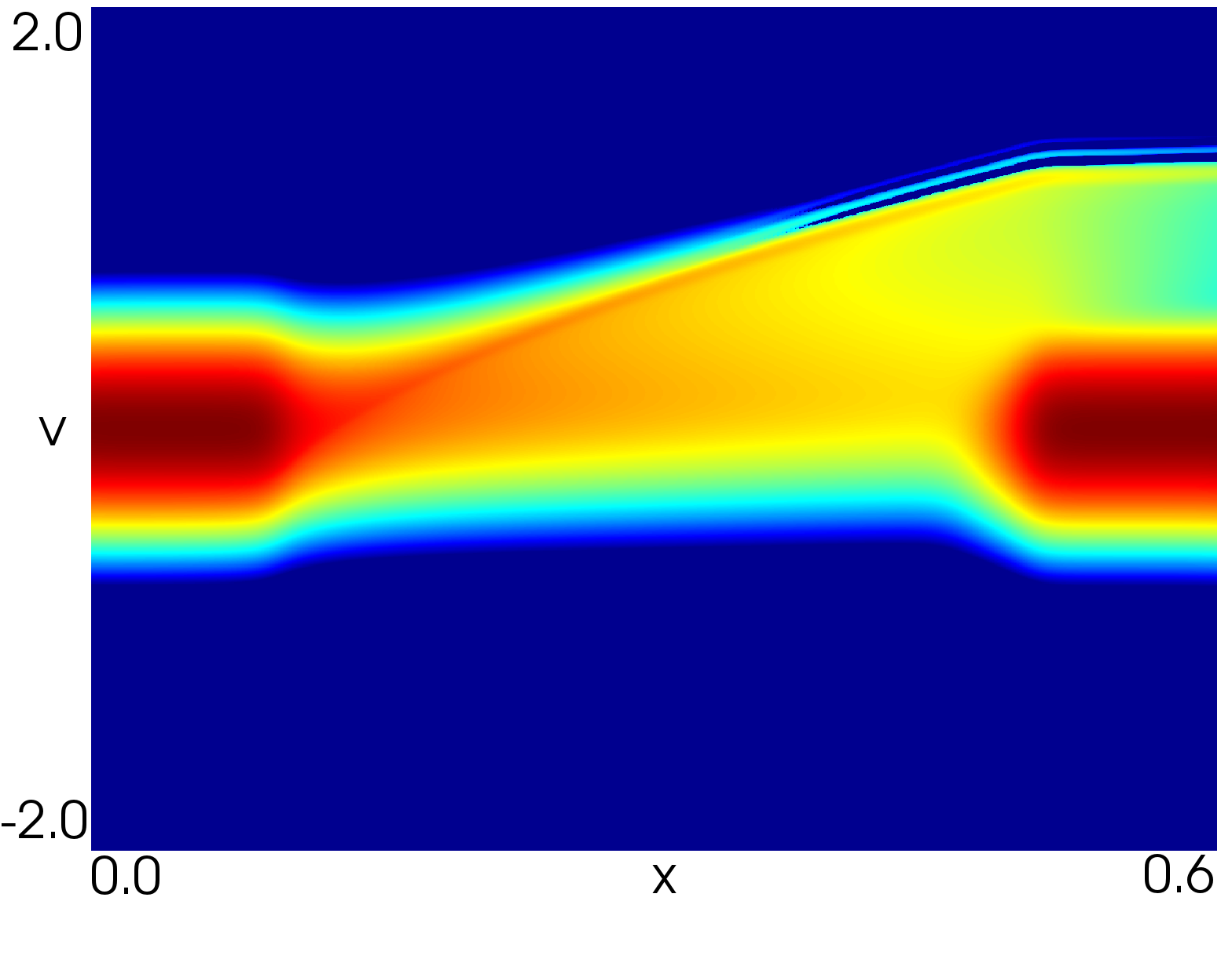}
		\caption*{Level 8}
	\end{subfigure}
	\begin{subfigure}{.32\textwidth}
		\includegraphics[width = \linewidth]{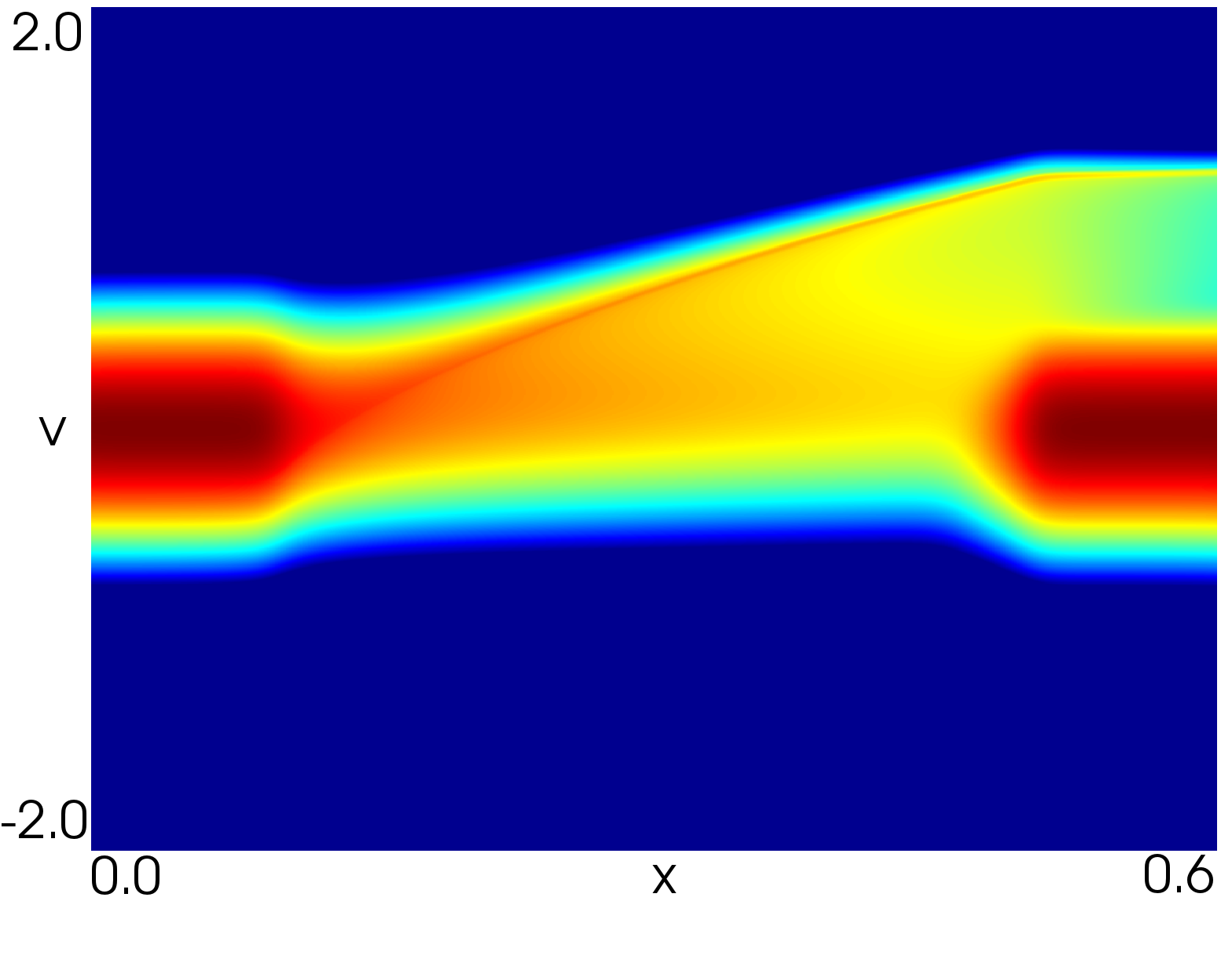}
		\caption*{Reference solution}
	\end{subfigure}
	\begin{subfigure}{.32\textwidth}
		\includegraphics[width=\linewidth]{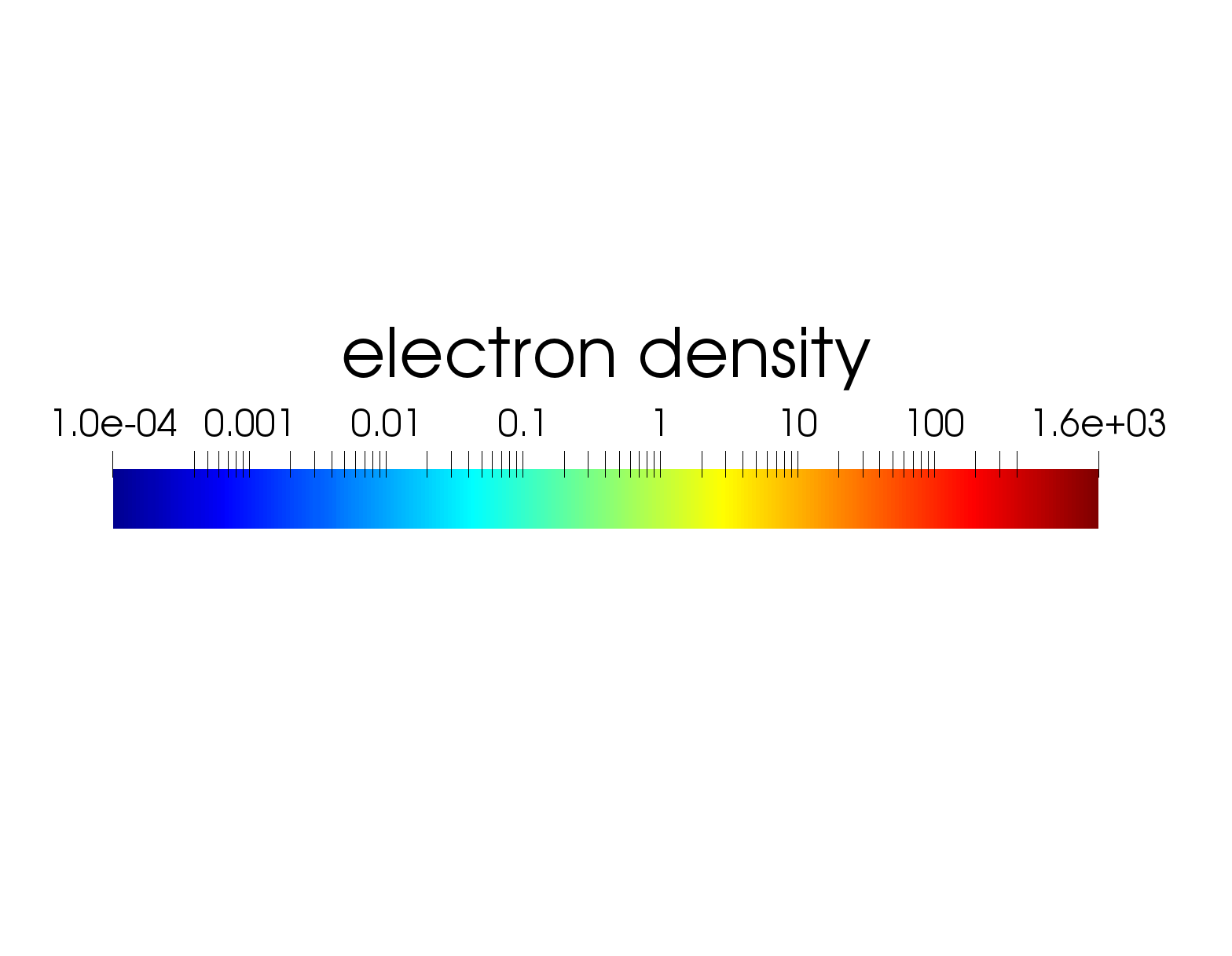}
		\caption*{color legend}
	\end{subfigure}
	\caption{Mesh convergence for the silocone diode benchmark problem in Section~\ref{sec:test_silicone_diode}. The solutions converge to the reference solution as the mesh is refined. $T_f = 0.05$. Level refers to the mesh parameters $(\Delta t, \Delta x, \Delta v) = 2^{\text{-Level}}(T_f, L, 2v_{\max})$.\label{fig:final_times}}
\end{figure}

The next test is a more extreme multiscale problem. This test uses the same problem configuration as the previous silicone diode problem, but with $\omega_{\min}=0.01$ in \eqref{eqn:omega_def}, and $\varepsilon = 0.002$. The results are shown in Table \ref{tab:multiscale}. Similar to the single scale case when $\varepsilon = 0.002$, we see a great improvement from using AA alone. 

\begin{table}[h]
\centering
\begin{tabular}{|l|ll|ll|ll|ll|}
\hline
  &   \multicolumn{2}{|c|}{NLS-AA} &   \multicolumn{2}{|c|}{NEST-AA} &   \multicolumn{2}{|c|}{NLS-PIC} &   \multicolumn{2}{|c|}{NEST-PIC}   \\
$\Delta t$ & time(s)  &   swps.    & time(s)  &   swps.    & time(s)  &   swps.    & time(s)  &   swps.     \\
\hline
\multicolumn{9}{|c|}{Without DDSA}\\
\hline
$T_f/2^{1}$ &    1.02 &    315  &    3.02 &   1166  &   15.45 &   5796  &   69.90 &  25303  \\
$T_f/2^{2}$ &    1.10 &    345  &    3.58 &   1379  &   17.28 &   6510  &   82.06 &  28327  \\
$T_f/2^{3}$ &    1.54 &    489  &    3.96 &   1537  &   21.20 &   7193  &   88.03 &  31220  \\
$T_f/2^{4}$ &    1.44 &    462  &    4.84 &   1795  &   21.60 &   8098  &   98.07 &  35039  \\
$T_f/2^{5}$ &    1.68 &    551  &    5.59 &   2202  &   24.66 &   9285  &  115.26 &  40200  \\
$T_f/2^{6}$ &    3.06 &   1009  &    7.13 &   2814  &   32.14 &  10810  &  164.26 &  46646  \\
$T_f/2^{7}$ &    2.99 &    993  &   10.03 &   3971  &   37.81 &  12673  &  204.00 &  54527  \\
$T_f/2^{8}$ &    4.27 &   1434  &   14.53 &   5772  &   40.06 &  14786  &  223.71 &  63175  \\
\hline
\multicolumn{9}{|c|}{With DDSA}\\
\hline
$T_f/2^{1}$ &    0.43 &    106  &    1.58 &    543 & \multicolumn{2}{c|}{INF}& \multicolumn{2}{c|}{FC} \\
$T_f/2^{2}$ &    0.50 &    134  &    1.77 &    652 & \multicolumn{2}{c|}{INF}& \multicolumn{2}{c|}{FC} \\
$T_f/2^{3}$ &    0.60 &    163  &    2.08 &    768 & \multicolumn{2}{c|}{INF}& \multicolumn{2}{c|}{FC} \\
$T_f/2^{4}$ &    0.73 &    207  &    2.85 &    962 & \multicolumn{2}{c|}{INF}& \multicolumn{2}{c|}{FC} \\
$T_f/2^{5}$ &    0.98 &    283  &    3.29 &   1251 & \multicolumn{2}{c|}{INF}& \multicolumn{2}{c|}{FC} \\
$T_f/2^{6}$ &    1.40 &    409  &    4.61 &   1757 & \multicolumn{2}{c|}{INF}& \multicolumn{2}{c|}{FC} \\
$T_f/2^{7}$ &    1.99 &    591  &    6.73 &   2583 & \multicolumn{2}{c|}{INF}& \multicolumn{2}{c|}{FC} \\
$T_f/2^{8}$ &    3.22 &    951  &   10.59 &   3880  &    9.11 &   2980  &   51.87 &  15511  \\
\hline
\end{tabular}
\caption{\textbf{Solver performance on the multiscale problem with $\varepsilon = 0.002$ and $\omega$ from \eqref{eqn:omega_def} with $\omega_{\min} = 0.01$.} The most efficient method for all $\Delta t$s was NLS-AA+DSA. \label{tab:multiscale}}
\end{table}
\subsection{Verification of convergence rate for NLS\label{sec:test_iter_rate}}

\begin{figure}[h]
\centering
\includegraphics[width=.495\linewidth]{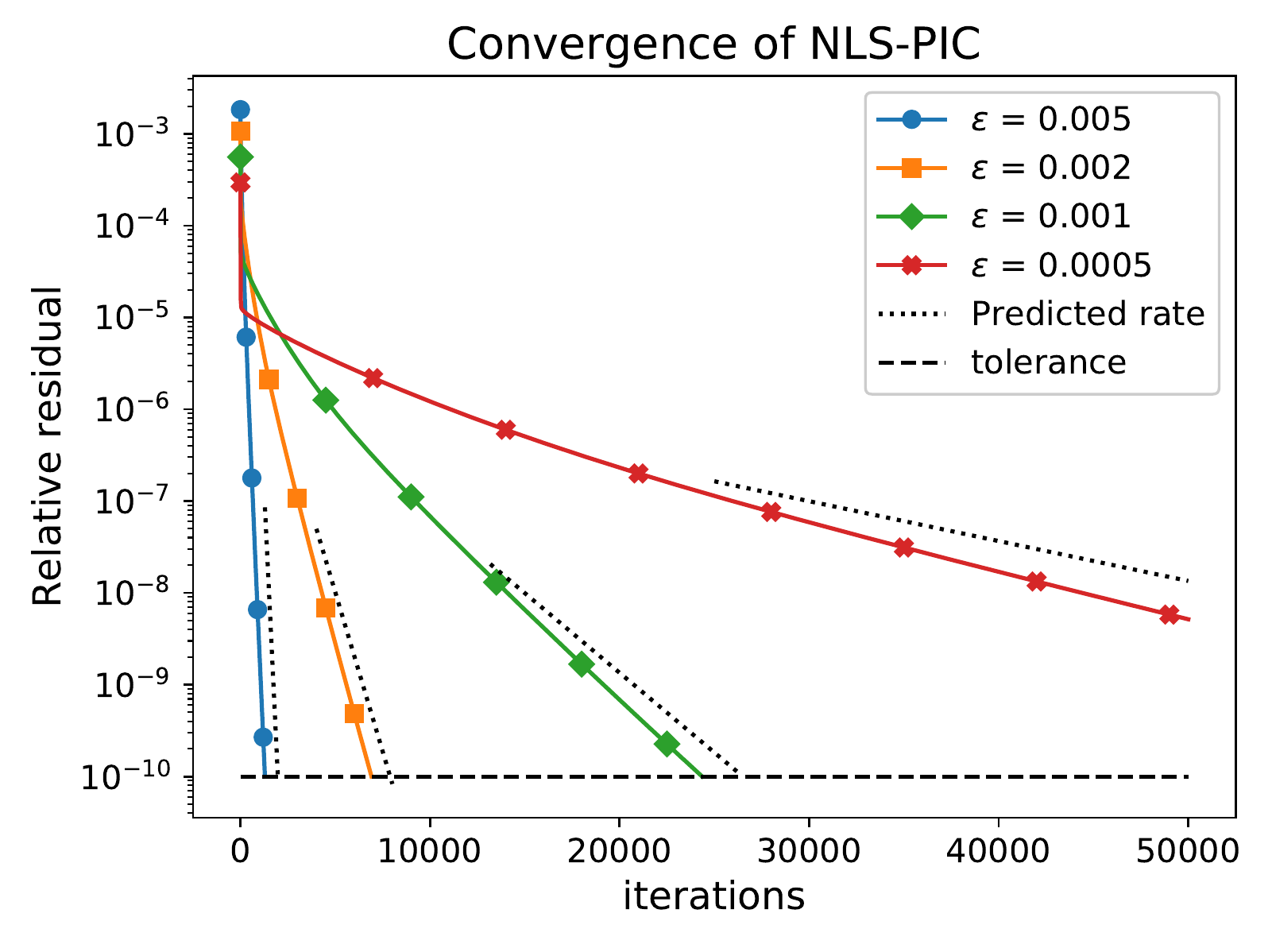}
\includegraphics[width=.495\linewidth]{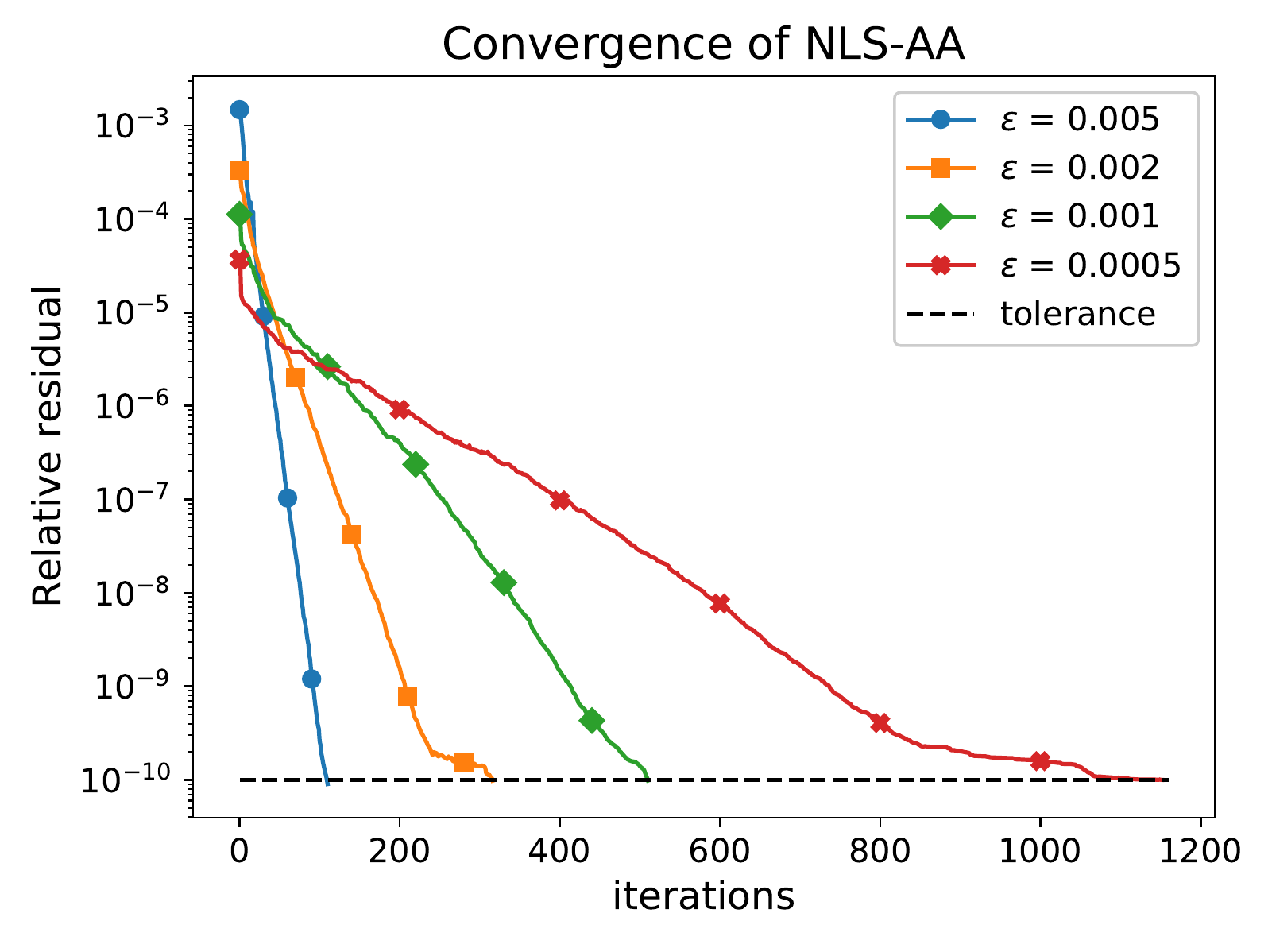}
\includegraphics[width=.495\linewidth]{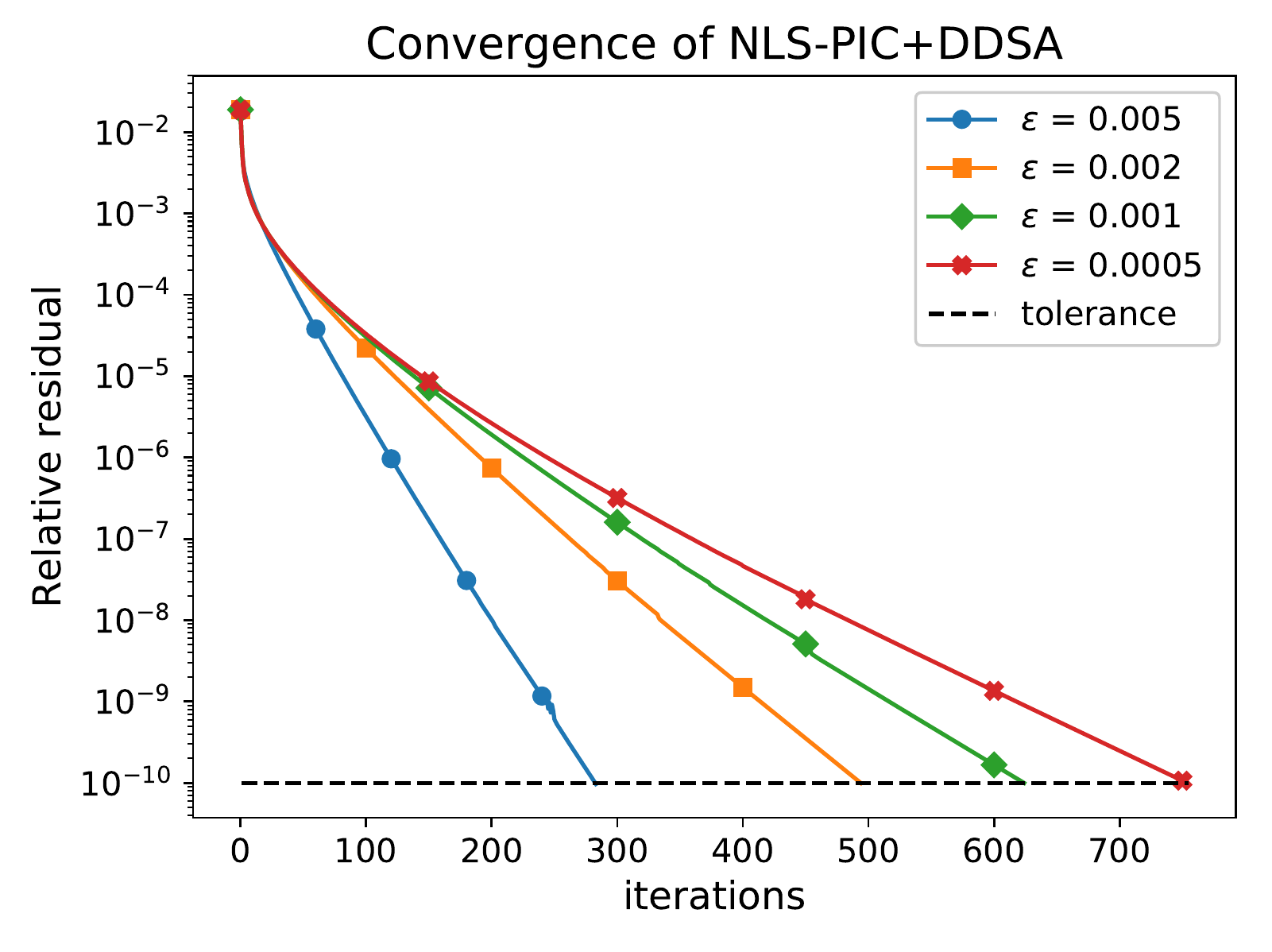}
\includegraphics[width=.495\linewidth]{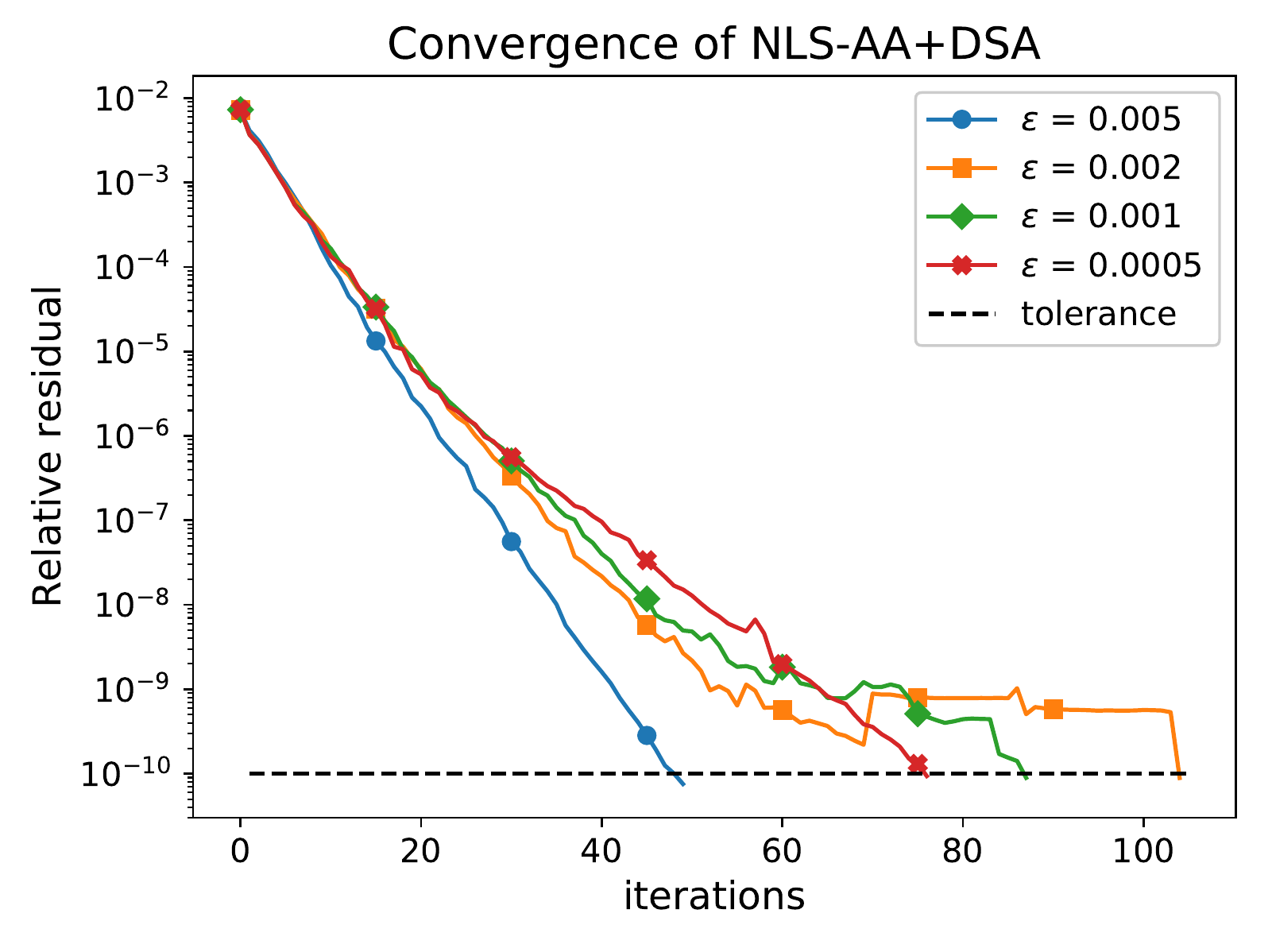}
\caption{Picard iterations converge at the predicted rate given by \eqref{eqn:contraction_constant_prediction}. The other methods shown all converge faster, with NLS-AA+DSA always being the fastest.\label{fig:iterative_convergence_rate}}
\end{figure}
\begin{table}[h]
\centering
\begin{tabular}{|l|l|ll|ll|ll|}
\hline
  & NLS-PIC & \multicolumn{2}{c|}{NLS-AA}  & \multicolumn{2}{c|}{NLS-PIC+DDSA}   & \multicolumn{2}{c|}{NLS-AA+DDSA}  \\
$\varepsilon$ & iterations & iterations & gain & iterations & gain  & iterations & gain \\
\hline
0.005  &  1293 & 112 & 11.5 & 285 & 4.5 & 51 & 25.4  \\
0.002  &  6914 & 318 & 21.7 & 496 & 13.9 & 106 & 65.2  \\
0.001  &  24357 & 512 & 47.6 & 626 & 38.9 & 89 & 273.7  \\
0.0005  &  84681 & 1154 & 73.4 & 756 & 112.0 & 78 & 1085.7  \\
\hline
\end{tabular}
\caption{Number of sweeps in first timestep. The gain represents the ratios of sweeps required by NLS-PIC compared to the method in the column. NLS-AA+DDSA required the fewest iterations, and required 1000 times fewer sweeps than NLS-PIC for the smallest $\varepsilon$. See Figure \ref{fig:iterative_convergence_rate} for the residual vs iteration.\label{tab:eff_gain}}
\end{table}

In this Section, we verify the analytic estimate of the spectral radius of NLS-PIC without DDSA from \eqref{eqn:contraction_constant_prediction} for several different $\varepsilon$. We use the same problem configuration as the single scale tests in Sections \ref{sec:test_single_scale} with the relative tolerance lowered to $10^{-10}$ and fixed timestep $\Delta t = 0.0025$. Only the first time step is used for this test, and it is taken with implicit Euler.  For these same parameters, we also apply AA, DDSA, and AA+DDSA. While NLS-AA still converges using the 200$\times$200 phase space mesh, we observe smoother convergence when the mesh is refined to 1000$\times$1000, which is used in all tests reported in this section. 

The relative residual versus iteration count is shown in Figure \ref{fig:iterative_convergence_rate}. A reference slope of the predicted convergence rate from \eqref{eqn:contraction_constant_prediction} is shown next to each Picard iteration experiment, and they show good agreement. The application of AA, DDSA, and AA+DDSA improves the convergence rates for all tested values of $\varepsilon$. AA+DDSA is always the most efficient, as shown in Table \ref{tab:eff_gain}.

\subsection{Verification of spatio-temporal convergence rate\label{sec:convergence_rate}}

In this section, we test the spatial and temporal convergence rates of the methods with implicit Euler and BDF2 using the NLS solver. We use a manufactured solution similar to one used for Vlasov-Possion simulations in \cite{ROSSMANITH20116203} and \cite{Garrett2018AFS}. To be consistent with the drift-diffusion limit, we enforce the solution to be symmetric in $v$ by writing the solution in terms of a Maxwellian, i.e.,
\begin{equation}
\label{eqn:man_soln}
f(x,v,t) = \frac{1}{2}\sqrt{\pi}\left( 2 - \cos(2x-2\pi t)  \right)M_{\frac{1}{8}}(v), 
\qquad 
E = -\frac{\sqrt{\pi}}{4}\sin(2x-2\pi t).
\end{equation}
\begin{table}[h]
\centering
\begin{tabular}{|l|ll|ll|ll|ll|}
\hline
 & \multicolumn{4}{c|}{Implicit Euler} & \multicolumn{4}{c|}{BDF2}\\
 \hline
 &  \multicolumn{2}{|c|}{$f$}  & \multicolumn{2}{|c|}{$E$}  & \multicolumn{2}{|c|}{$f$}  & \multicolumn{2}{|c|}{$E$} \\ 
$\mathbf{L}$ & error & rate & error & rate & error & rate & error & rate \\ 
\hline
\multicolumn{9}{|c|}{$\varepsilon = 1$}\\
\hline
1 & 1.72E-01  &  ---  & 2.78E-01  &  ---  & 1.39E-01  &  ---  & 2.15E-01  & ---  \\
2 & 8.13E-02  &  1.08  & 1.42E-01  &  0.97  & 3.89E-02  &  1.84  & 5.73E-02  &  1.91 \\
3 & 3.99E-02  &  1.03  & 7.21E-02  &  0.98  & 9.96E-03  &  1.97  & 1.46E-02  &  1.97 \\
4 & 1.99E-02  &  1.00  & 3.63E-02  &  0.99  & 2.50E-03  &  1.99  & 3.68E-03  &  1.99 \\
5 & 9.98E-03  &  1.00  & 1.82E-02  &  0.99  & 6.26E-04  &  2.00  & 9.20E-04  &  2.00 \\
6 & 5.00E-03  &  1.00  & 9.13E-03  &  1.00  & 1.57E-04  &  2.00  & 2.30E-04  &  2.00\\
\hline
\multicolumn{9}{|c|}{$\varepsilon = 0.001$}\\
\hline
1 & 2.45E-01  &  ---  & 3.52E-01  &  ---  & 2.27E-01  &  ---  & 2.42E-01  & ---  \\
2 & 6.21E-02  &  1.98  & 1.65E-01  &  1.09  & 3.52E-02  &  2.69  & 6.21E-02  &  1.96 \\
3 & 2.88E-02  &  1.11  & 8.31E-02  &  0.99  & 8.97E-03  &  1.97  & 1.59E-02  &  1.97 \\
4 & 1.41E-02  &  1.03  & 4.16E-02  &  1.00  & 2.25E-03  &  2.00  & 3.98E-03  &  2.00 \\
5 & 7.01E-03  &  1.01  & 2.08E-02  &  1.00  & 5.63E-04  &  2.00  & 9.94E-04  &  2.00 \\
6 & 3.50E-03  &  1.00  & 1.04E-02  &  1.00  & 1.41E-04  &  2.00  & 2.47E-04  &  2.01 \\
\hline 
\end{tabular}
\caption{\label{tab:temporal_rates}$\mathbf{L}$ represents the refinement level. The overall $L^2$ errors are first order for implicit Euler, and second order for BDF2. The errors are bounded independent of $\varepsilon$.}
\end{table}
The parameters in \eqref{eqn:model} are $\omega=1$, $D = \sqrt{\pi}$ and $\Theta = 1/8$. We also use an $\varepsilon$-dependent source term to yield the manufactured solution:
\begin{equation}
q(x,v,t) = \varepsilon^{-1}2 \exp( -4v^2) \sin\left(2 x - 2 \pi t \right) \left( v - \veps \pi + v\sqrt{\pi}\left(2-\cos(2 x - 2\pi t) \right) \right).
\end{equation}
In addition, the computational domain has been changed to  $\comp = [-\pi,\pi]^2$, and we use periodic boundary conditions at $x = \pm \pi$ and zero inflow boundary conditions at $v = \pm \pi$, since the exact solution approaches zero rapidly in $v$. To implement periodic boundary conditions, we slightly modify Algorithm $\ref{alg:nls}$ by appending the inflow boundary condition of $f_h$ to the vector $(\rho_h, \hat{f}_h)$, and setting the inflow values equal to the outflow values at the same $v$ coordinate from the previous iteration.   (See \cite{Garrett2018AFS} for details.)

We initialize the numerical solution $\Pi_h$ applied to the exact solution \eqref{eqn:man_soln} at $t=0$ and measure the relative error at a final time of $T_f=1$ for both the electron density $f_h$ and the electric field $E_h$. The error for $f_h$ is measured in the $L^2(\comp)$ norm, and the error for $E_h$ is measured in the $L^2(X)$ norm.  
We compute the solution at different refinement levels by successively doubling the number of cells in the $x$ and $v$ directions while halving the size of the timesteps. We perform this convergence study for $\varepsilon = 1$ and $\varepsilon = 10^{-3}$. The relative tolerance of NLS is lowered to $10^{-10}$ because a relative tolerance of $10^{-8}$ was not sufficient to obtain second order convergence when $\varepsilon = 10^{-3}$ (see the discussion about false convergence in Section~\ref{sec:test_disc_details}). 

The results reported in Table \ref{tab:temporal_rates} show that, for both $f_h$ and $E_h$, the convergence rates of implicit Euler and BDF2 are first- and second-order, respectively. This indicates that a piecewise constant electric field may be sufficient for overall second-order convergence, although a rigorous analysis is still required.

\section{Conclusions}

We have derived a new energy-based proof of stability for implicit  time discretizations for a simplified Boltzmann-Poisson model.  At the continuum level, the  proof establishes an $\varepsilon$-independent growth factor, with weights in the $L^2$ energy that do not depend on time.  Thus, we can apply standard $G$-Stability theory for linear multistep methods and guarantee stability under an $O(1)$ timestep restriction. 

We have also proposed a new iterative solver, NLS, and proved its convergence under the assumption of a fixed electric field. The task of proving convergence of this solver with a self-consistent electric field is left for future work. We have derived an accelerator using the drift-diffusion limit as a low order model for correcting the NLS solver. We have demonstrated numerically that the NLS-based methods are more efficient across a range of problems than one previously developed in \cite{laiu2019fast}. 

Other future work includes testing the NLS solver on higher dimensions, exploring how to parallelize it, and modifying the sweeping procedure to allow the use of higher order approximations of the  electric field.   In addition, we intend to embed the NLS solver into a hybrid formulation originally developed for radiation transport problems \cite{hauck2013collision,crockatt2017arbitrary,crockatt2019hybrid,heningburg2020hybrid,crockatt2020improvements}.

\bibliographystyle{siamplain}
\bibliography{semiconductor.bib}

\begin{thebibliography}{10}

\bibitem{abdallah2004}
{\sc N.~B. Abdallah and M.~L. Tayeb}, {\em Diffusion approximation for the one
  dimensional {B}oltzmann-{P}oisson system}, Discrete \& Continuous Dynamical
  Systems - B, 4 (2004), p.~1129,
  \url{https://doi.org/10.3934/dcdsb.2004.4.1129},
  \url{http://aimsciences.org//article/id/c6210964-74b3-40c1-98aa-be43485b3dbe}.

\bibitem{Adams-Larsen-2002}
{\sc M.~L. Adams and E.~W. Larsen}, {\em Fast iterative methods for
  discrete-ordinates particle transport calculations}, Progress in Nuclear
  Energy, 40 (2002), pp.~3 -- 159,
  \url{https://doi.org/https://doi.org/10.1016/S0149-1970(01)00023-3},
  \url{http://www.sciencedirect.com/science/article/pii/S0149197001000233}.

\bibitem{Alcouffe-1976}
{\sc R.~E. Alcouffe}, {\em A stable diffusion synthetic acceleration method for
  neutron transport iterations}, Trans. Am. Nucl. Soc., 23 (1976).

\bibitem{Alcouffe-1977}
{\sc R.~E. Alcouffe}, {\em Diffusion synthetic acceleration methods for the
  diamond-differenced {D}iscrete-{O}rdinates equations}, Nuclear Science and
  Engineering, 64 (1977), pp.~344--355, \url{https://doi.org/10.13182/NSE77-1},
  \url{https://doi.org/10.13182/NSE77-1},
  \url{https://arxiv.org/abs/https://doi.org/10.13182/NSE77-1}.

\bibitem{ayuso2011discontinuous}
{\sc B.~Ayuso, J.~A. Carrillo, C.-W. Shu, et~al.}, {\em Discontinuous
  {G}alerkin methods for the one-dimensional {V}lasov-{P}oisson system},
  Kinetic and Related Models, 4 (2011), pp.~955--989.

\bibitem{BARTH20063311}
{\sc T.~Barth}, {\em On discontinuous {G}alerkin approximations of {B}oltzmann
  moment systems with {L}evermore closure}, Computer Methods in Applied
  Mechanics and Engineering, 195 (2006), pp.~3311 -- 3330,
  \url{https://doi.org/https://doi.org/10.1016/j.cma.2005.06.016},
  \url{http://www.sciencedirect.com/science/article/pii/S0045782505002719}.
\newblock Discontinuous Galerkin Methods.

\bibitem{cercignani2000}
{\sc C.~Cercignani, I.~M. Gamba, J.~W. Jerome, and C.-W. Shu}, {\em A domain
  decomposition method for silicon devices}, Transport Theory and Statistical
  Physics, 29 (2000), pp.~525--536,
  \url{https://doi.org/10.1080/00411450008205889},
  \url{https://doi.org/10.1080/00411450008205889},
  \url{https://arxiv.org/abs/https://doi.org/10.1080/00411450008205889}.

\bibitem{crockatt2017arbitrary}
{\sc M.~M. Crockatt, A.~J. Christlieb, C.~K. Garrett, and C.~D. Hauck}, {\em An
  arbitrary-order, fully implicit, hybrid kinetic solver for linear radiative
  transport using integral deferred correction}, Journal of Computational
  Physics, 346 (2017), pp.~212--241.

\bibitem{crockatt2019hybrid}
{\sc M.~M. Crockatt, A.~J. Christlieb, C.~K. Garrett, and C.~D. Hauck}, {\em
  Hybrid methods for radiation transport using diagonally implicit runge--kutta
  and space--time discontinuous galerkin time integration}, Journal of
  Computational Physics, 376 (2019), pp.~455--477.

\bibitem{crockatt2020improvements}
{\sc M.~M. Crockatt, A.~J. Christlieb, and C.~D. Hauck}, {\em Improvements to a
  class of hybrid methods for radiation transport: Nystr{\"o}m reconstruction
  and defect correction methods}, Journal of Computational Physics,  (2020),
  p.~109765.

\bibitem{dahlquist1978}
{\sc G.~Dahlquist}, {\em {G}-stability is equivalent to {A}-stability}, BIT
  Numerical Mathematics, 18 (1978), pp.~384--401.

\bibitem{dimarco2014implicit}
{\sc G.~Dimarco, L.~Pareschi, and V.~Rispoli}, {\em Implicit-{E}xplicit
  {R}unge-{K}utta schemes for the {B}oltzmann-{P}oisson system for
  semiconductors}, Communications in Computational Physics, 15 (2014),
  pp.~1291--1319.

\bibitem{emmrichdiscrete}
{\sc E.~Emmrich}, {\em Discrete versions of {G}ronwall's lemma and their
  application to the numerical analysis of parabolic problems}.
\newblock preprint on webpage at
  \url{https://www.math.tu-berlin.de/fileadmin/i26_fg-emmrich/Publikationen/Preprints_and_submitted_papers/Emmrich/Emmrich1999_GronwallsLemma.pdf},
  July 1999.

\bibitem{evans2018proof}
{\sc C.~Evans, S.~Pollock, L.~G. Rebholz, and M.~Xiao}, {\em A proof that
  {A}nderson acceleration improves the convergence rate in linearly converging
  fixed point methods (but not in those converging quadratically)}, 2018,
  \url{https://arxiv.org/abs/1810.08455}.

\bibitem{Garrett2018AFS}
{\sc C.~K. Garrett and C.~D. Hauck}, {\em A fast solver for implicit
  integration of the {V}lasov-{P}oisson system in the {E}ulerian framework},
  SIAM J. Scientific Computing, 40 (2018).

\bibitem{eigenweb}
{\sc G.~Guennebaud, B.~Jacob, et~al.}, {\em Eigen v3}.
\newblock http://eigen.tuxfamily.org, 2010.

\bibitem{hauck2013collision}
{\sc C.~D. Hauck and R.~G. McClarren}, {\em A collision-based hybrid method for
  time-dependent, linear, kinetic transport equations}, Multiscale Modeling \&
  Simulation, 11 (2013), pp.~1197--1227.

\bibitem{heningburg2020hybrid}
{\sc V.~Heningburg and C.~D. Hauck}, {\em Hybrid solver for the radiative
  transport equation using finite volume and discontinuous galerkin}, arXiv
  preprint arXiv:2002.02517,  (2020).

\bibitem{hu2014}
{\sc Z.~Hu, R.~Li, T.~Lu, Y.~Wang, , and W.~Yao}, {\em Simulation of an
  n$^+$-n-n$^+$ diode by using globally-hyperbolically-closed high-order moment
  models}, Journal of Scientific Computing, 59 (2014), pp.~761--774.

\bibitem{jin2000discretization}
{\sc S.~Jin and L.~Pareschi}, {\em Discretization of the multiscale
  semiconductor boltzmann equation by diffusive relaxation schemes}, Journal of
  Computational Physics, 161 (2000), pp.~312--330.

\bibitem{laiu2019fast}
{\sc M.~P. Laiu, Z.~Chen, and C.~D. Hauck}, {\em A fast implicit solver for
  semiconductor models in one space dimension}, Journal of Computational
  Physics, 417 (2020), p.~109567,
  \url{https://doi.org/https://doi.org/10.1016/j.jcp.2020.109567},
  \url{http://www.sciencedirect.com/science/article/pii/S0021999120303417}.

\bibitem{larsen2010advances}
{\sc E.~W. Larsen and J.~E. Morel}, {\em Advances in discrete-ordinates
  methodology}, in Nuclear Computational Science, Springer, 2010, pp.~1--84.

\bibitem{Levermore1998MomentCH}
{\sc C.~D. Levermore}, {\em Moment closure hierarchies for the
  {B}oltzmann-{P}oisson equation}, VLSI Design, 1998 (1998), pp.~97--101.

\bibitem{Liu2010}
{\sc H.~Liu and J.~Yan}, {\em The direct discontinuous galerkin ({DDG}) method
  for diffusion with interface corrections}, Commun. Comput. Phys, 8 (2010),
  pp.~541--564, \url{https://doi.org/10.4208/cicp.010909.011209a}.

\bibitem{markowich2012semiconductor}
{\sc P.~A. Markowich, C.~A. Ringhofer, and C.~Schmeiser}, {\em Semiconductor
  equations}, Springer Science \& Business Media, 2012.

\bibitem{masmoudi2007diffusion}
{\sc N.~Masmoudi and M.~L. Tayeb}, {\em Diffusion limit of a semiconductor
  {B}oltzmann--{P}oisson system}, SIAM journal on mathematical analysis, 38
  (2007), pp.~1788--1807.

\bibitem{ni2010linearly}
{\sc P.~Ni and H.~F. Walker}, {\em A linearly constrained least-squares problem
  in electronic structure computations}, ICCES. v7 i1,  (2010), pp.~43--49.

\bibitem{pollock2019anderson}
{\sc S.~Pollock and L.~Rebholz}, {\em Anderson acceleration for contractive and
  noncontractive operators}, 2019, \url{https://arxiv.org/abs/1909.04638}.

\bibitem{ringhofer2002}
{\sc C.~Ringhofer}, {\em Numerical methods for the semiconductor {B}oltzmann
  equation based on spherical harmonics expansions and entropy
  discretizations}, Transport Theory and Statistical Physics, 3131 (2002),
  pp.~41--1450, \url{https://doi.org/10.1081/TT-120015508}.

\bibitem{ROSSMANITH20116203}
{\sc J.~A. Rossmanith and D.~C. Seal}, {\em A positivity-preserving high-order
  semi-{L}agrangian discontinuous {G}alerkin scheme for the
  {V}lasov–{P}oisson equations}, Journal of Computational Physics, 230
  (2011), pp.~6203 -- 6232,
  \url{https://doi.org/https://doi.org/10.1016/j.jcp.2011.04.018},
  \url{http://www.sciencedirect.com/science/article/pii/S0021999111002579}.

\bibitem{saad1986gmres}
{\sc Y.~Saad and M.~H. Schultz}, {\em {GMRES}: A generalized minimal residual
  algorithm for solving nonsymmetric linear systems}, SIAM Journal on
  scientific and statistical computing, 7 (1986), pp.~856--869.

\bibitem{schmeiser1998convergence}
{\sc C.~Schmeiser and A.~Zwirchmayr}, {\em Convergence of moment methods for
  linear kinetic equations}, SIAM journal on numerical analysis, 36 (1998),
  pp.~74--88.

\bibitem{walker2011anderson}
{\sc H.~F. Walker and P.~Ni}, {\em Anderson acceleration for fixed-point
  iterations}, SIAM Journal on Numerical Analysis, 49 (2011), pp.~1715--1735.

\bibitem{wanner1996solving}
{\sc G.~Wanner and E.~Hairer}, {\em Solving ordinary differential equations
  II}, Springer Berlin Heidelberg, 1996.

\end{thebibliography}
\pagebreak

\end{document}